\documentclass{article}
\usepackage[utf8]{inputenc}

\usepackage{amsmath}
\usepackage{amssymb}
\usepackage{amsthm}
\usepackage{amsfonts}
\usepackage{graphicx}
\usepackage{multirow}
\usepackage{xcolor}

\usepackage{subcaption}

\newtheorem{theorem}{Theorem}
\newtheorem{lemma}{Lemma}
\newtheorem{assumption}{Assumption}

\usepackage{natbib}

\usepackage{xcolor}

\usepackage{authblk}
\title{Measure of shape for object data}

\author{Joni Virta}
\affil{Department of Mathematics and Statistics\\ University of Turku}

\date{}

\begin{document}

\maketitle

\begin{abstract}
Object data analysis is concerned with statistical methodology for datasets whose elements reside in an arbitrary, unspecified metric space. In this work we propose the object shape, a novel measure of shape/symmetry for object data. The object shape is easy to compute and interpret, owing to its intuitive interpretation as interpolation between two extreme forms of symmetry. As one major part of this work, we apply object shape in various metric spaces and show that it manages to unify several pre-existing, classical forms of symmetry. We also propose a new visualization tool called the peeling plot, which allows using the object shape for outlier detection and principal component analysis of object data.
\end{abstract}

\section{Introduction}

Modern applications routinely produce datasets, such as images, functions or graphs, that do not take the familiar form of $n \times p$ data matrices. An emerging trend in the literature is to approach the analysis of such data in a type-agnostic way, by not specifying the actual sample space, but simply assuming that the observed \textit{objects}, $X_1, \ldots, X_n$, reside in a general metric space $(\mathcal{X}, d)$. Statistical methodology which depends on the object sample only through the interobject distances, $d(X_i, X_j)$, is collectively known as object data analysis, see, e.g., \cite{bhattacharya2003large, lyons2013distance, dubey2019frechet, dubey2022modeling, virta2022sliced, virta2023spatial} for works taking this viewpoint. The obvious advantage of such an approach is that any object data method is automatically applicable to all forms of object data, be they images, correlation matrices or graphs, making object data analysis highly universal.

The purpose of the current work is to propose a new descriptive statistic, the \textit{object shape}, for measuring the shape (level of symmetry), of object data. Given a distribution $P$ taking values in a fixed metric space $(\mathcal{X}, d)$ we define the object shape of $P$ as 
\begin{align}\label{eq:main_concept}
    O(P) = \frac{E \{ d(X_1, X_2)^2 d(X_1, X_3)^2 \}}{E \{ d(X_1, X_2)^4 \}},
\end{align}
where $X_1, X_2, X_3$ are independent draws from $P$. While the quantity $O(P)$ appears rather simple, we show in the sequel that it both has a very intuitive interpretation and manages to unify several well-known, established forms of shape/symmetry. To be more specific, as our main contributions about $O(P)$ we show that:

\begin{itemize}
    \item[(i)] The object shape $O(P)$ takes values in the interval $[1/2, 1]$. Moreover, $O(P) = 1/2$ if and only if the distribution $P$ is concentrated on a line (in the metric $d$) in $\mathcal{X}$, and $O(P) = 1$ if and only if three i.i.d. objects drawn from $P$ almost surely form an equilateral triangle (in the metric $d$). In other words, the object shape interpolates between two extreme forms of symmetry: the smaller $O(P)$ is, the more one-dimensional the distribution $P$, and the larger $O(P)$ is, the more evenly spread out $P$ is in the space~$\mathcal{X}$.
    \item[(ii)] The object shape $O(P)$ reduces to well-known measures of shape/symmetry under specific choices of $(\mathcal{X}, d)$ and $P$. For example, (a) for elliptical distributions in a Euclidean space, the object shape coincides with the classical notion of sphericity: the larger $O(P)$ is, the more the equidensity contours of $P$ resemble spheres, (b) for discrete distributions over a finite set, the object shape reduces to measuring the uniformity of the distribution $P$, as in the classical Pearson's $\chi^2$-test. Further examples are given in Section~\ref{sec:cases}.
    \item[(iii)] The scenario-specific upper bounds for the object shape $O(P)$ can be used to devise hypothesis tests for the related symmetries. For example, besides recovering the Pearson's $\chi^2$-test for discrete data as mentioned already above, we obtain a new test of uniformity on the unit circle, see Section \ref{sec:cases}.
\end{itemize}

The diversity and difficulty of visualizing object data means that data summaries such as the proposed object shape can be seen to be even more important for object data than in traditional data analysis. Thus far, the literature on descriptive tools for object data has focused heavily on location estimation, particularly in conjunction with statistical depth measures, see \cite{cholaquidis2020weighted, dai2022tukey, geenens2023statistical, virta2023spatial}. The concept of Frech{\'e}t mean, i.e., the minimizer of the map $\mu \mapsto E\{ d(X, \mu)^2 \}$, also falls in this category, see, e.g., \cite{bhattacharya2003large, dubey2019frechet}. The most prominent examples of descriptive statistics measuring (co)variation are the Frech{\'e}t variance, i.e.,  $E\{ d(X, \mu_0)^2 \}$ where $\mu_0$ is the Frech{\'e}t mean, and the distance covariance \citep{lyons2013distance}. In some sense multidimensional scaling \citep{kruskal1978multidimensional}, a classical method of object data visualization, can also be seen as a member of this class as it is essentially based on principal component analysis, a second-order method. 

Using classical statistical terminology, none of the methods listed in the previous paragraph thus measure \textit{shape}, i.e., the properties of $P$ beyond location or scale. Whereas, as we argue in Section \ref{sec:main}, our object shape $O(P)$ is invariant to both location and scale (at least when $\mathcal{X}$ is structured enough to admit such concepts), meaning that calling it a measure of shape is warranted. Indeed, as far as we are aware, our object shape is the first measure of shape for object data. This, in conjunction with the clear interpretation and fast computation of $O(P)$ in practice, means that our proposal offers a valuable addition to the toolbox of descriptive statistical analysis of object data. 

In addition to establishing the theoretical properties of $O(P)$ listed earlier, as one of our contributions we show how object shape gives rise to a new visualization tool for object data which we call the peeling plot. This plot is constructed by removing observations one-by-one from an object data set in such a way that the object shape $O(P)$ is maximized/minimized at each step. The peeling plot is then obtained as the scatter plot between the indices of the removed observations and the value of $O(P)$ at each step. If minimization is used, the peeling plot leads into a novel method for conducting principal component analysis for object data. Whereas, under maximization, the peeling plot can be used for detecting outlying objects, see the examples in Section \ref{sec:peeling}. 


\section{Main result}\label{sec:main}

Let $(\mathcal{X}, d)$ be a metric space and let $\mathcal{P}$ denote a probability distribution on $X$. Throughout this work, we assume that the following condition holds true.

\begin{assumption}\label{assu:finite_fourth_moment}
The distribution $P$ is such that, for independent $X_1, X_2 \sim P$,
\begin{itemize}
    \item[(i)] There exists a point $a \in \mathcal{X}$ such that $E \{ d(X_1, a)^4 \} < \infty$.
    \item[(ii)] The random variable $d(X_1, X_2)$ is not almost surely equal to zero.
\end{itemize}

\end{assumption}

Assumption \ref{assu:finite_fourth_moment}(i) can be seen as the object data equivalent of assuming that the fourth moment of a real random variable is finite. In particular, by Cauchy-Schwarz inequality and the triangle inequality, it guarantees that also the moments $E \{ d(X_1, X_2)^2 d(X_1, X_3)^2 \}$ and $E \{ d(X_1, X_2)^4 \} $ exist as finite. Whereas, Assumption \ref{assu:finite_fourth_moment}(ii) simply requires that the distribution $\mathcal{P}$ is not a trivial Dirac point mass.

Our main point of interest is the object shape $O(P)$, defined in \eqref{eq:main_concept}, which is well-defined under Assumption \ref{assu:finite_fourth_moment}. Before its proper study, the form of $O(P)$ already offers us some hints on its meaning and interpretation: Firstly, the denominator and numerator in \eqref{eq:main_concept} have the same ``degree'' (four), implying that $O(P)$ is a dimensionless quantity. Secondly, if $(\mathcal{X}, d)$ is a normed space such that $d(X_1, X_2) = \| X_1 - X_2 \|$ for some norm $\| \cdot \|$, it is clear that $O(P)$ is invariant to translation and scaling of the distribution $P$ (i.e., invariant to maps $X_i \mapsto a X_i + b$). These two observations together lead us to conclude that $O$ measures some deeper, scale and location-invariant aspect of the distribution $P$, i.e., its shape.

As our main result of this section, we next show that $O(P)$ is constrained to the interval $[1/2, 1]$ and give geometric characterizations for the endpoints of this interval. In the following theorem we use the phrase  ``$x_1, x_2, x_3$ reside on a line'', to mean that the three points are such that the triangle inequality $d(x_i, x_j) \leq d(x_i, x_j) + d(x_j, x_k)$ achieves equality for some ordering of the indices 1, 2, 3. In a Euclidean space, this terminology thus corresponds to the standard definition of lines.

\begin{theorem}\label{theo:geometric_characterization}
Under Assumption \ref{assu:finite_fourth_moment}, $O(P) \in [1/2, 1]$. Furthermore,
\begin{itemize}
    \item [(i)] $O(P) = 1/2$ if and only if
        \begin{align*}
        \mathrm{pr}(\{x_1, x_2, x_3 \in \mathcal{X} \mid x_1, x_2, x_3 \mbox{ reside on a line.} \}) = 1.
    \end{align*}
    \item [(ii)] $O(P) = 1$ if and only if
    \begin{align*}
        \mathrm{pr}(\{x_1, x_2, x_3 \in \mathcal{X} \mid d(x_1, x_2) = d(x_2, x_3) = d(x_3, x_1) \}) = 1.
    \end{align*}
\end{itemize}
\end{theorem}

By Theorem \ref{theo:geometric_characterization}, it is indeed reasonable to regard $O(P)$ as a a measure of symmetry or shape of the distribution $P$. The quantity $O(P)$ achieves its maximal value if and only if three points drawn i.i.d. from $P$ almost surely form an equilateral triangle, in the sense of the metric $d$. While such an event is unlikely in any practically reasonable scenario, intuitively it means that large values of $O(P)$ can only be achieved when $P$ is spread out sufficiently evenly in the space $\mathcal{X}$. Later in Section \ref{sec:cases} we show that in several situations $O$ is maximized when the distribution in question is ``uniform'' in some specific sense, implying that the previous intuition indeed holds true. Whereas, the minimal value is achieved precisely when $P$ is almost surely concentrated on a line (in the sense of the metric $d$). Hence, $O(P)$ can also be seen as a continuous relaxation of the concept of dimension: its minimal value corresponds to $P$ charging all its mass on a ``line'', i.e., a one-dimensional object, in $\mathcal{X}$, whereas $O(P)$ is maximized when $P$ spreads out to the full space $\mathcal{X}$, corresponding, in some sense, to a maximal dimension.

Being essentially a moment-based quantity, the sample estimation of \eqref{eq:main_concept} is simple. For a distribution $P$, we let the notation $P_n$ denote the empirical distribution of a random sample $X_1, \ldots, X_n$ of size $n$ drawn from $P$. Hence, a natural estimator of $O(P)$ is
\begin{align}\label{eq:main_concept_sample}
    O(P_n) = \frac{\frac{1}{n^3} \sum_{i = 1}^n \sum_{j = 1}^n \sum_{k = 1}^n d(X_i, X_j)^2 d(X_i, X_k)^2}{\frac{1}{n^2} \sum_{i = 1}^n \sum_{j = 1}^n \sum_{k = 1}^n d(X_i, X_j)^4} = \frac{1_n^{\mathrm{T}} B_n^2 1_n}{n \| B_n \|^2},
\end{align}
where $1_n \in \mathbb{R}^n$ is a vector of ones, the $n \times n$ matrix $B_n = \{ d(X_i, X_j)^2 \}$ contains all pairwise squared distances between the observations and $\| \cdot \|$ denotes the Frobenius norm. Consequently, given a matrix $B_n$, computing $O(P_n)$ is an operation of complexity $\mathcal{O}(n^3)$. By the standard results on $U$-statistics \citep{lee1990u}, it is simple to check that this estimator is consistent, $O(P_n) = O(P) + o_P(1)$. Furthermore, the convergence rate of root-$n$ can be obtained for the error term by taking on stronger moment conditions for $d(X_1, a)$ in Assumption~\ref{assu:finite_fourth_moment}.

\section{Example scenarios}\label{sec:cases}

\subsection{Introduction}

In the following subsections, we illustrate the object shape under four different combinations of metric space and distributional family. In each case, we show that the range of values of $O(P)$ is actually a strict subset of the general interval $[1/2, 1]$ derived in Theorem \ref{theo:geometric_characterization}. Additionally, we characterize the minimal and maximal values of $O(P)$ in terms of the parameters of the underlying distribution family and derive a hypothesis test of symmetry corresponding to the distribution with the largest value for $O(P)$ in each case. We stress that not all of these scenarios are necessarily particularly ``relevant'' in modern statistical applications. Instead, they have been chosen here (i) to demonstrate the wide range of situations $O(P)$ can be applied in, (ii) for the clear interpretations they offer, and (iii) because closed-form solutions are available in them.  

We are also aware that there are more specialized tools available for studying each of these data scenarios. Indeed, our task here is not to derive the most efficient solutions or tests for each given scenario, but rather to show that multiple seemingly disjoint concepts in classical statistics are unified by and expressible via the object shape $O(P)$, further highlighting its usefulness in the analysis of object data. 

\subsection{Elliptical distribution}

In this subsection, we take $(\mathcal{X}, d)$ to be the $p$-dimensional Euclidean space equipped with the Euclidean metric. Let $R \sim F_R$ where $F_R$ is some fixed distribution on the non-negative real numbers. Let $U$ be independent of $R$ and obey the uniform distribution on the unit sphere $\mathbb{S}^{p - 1}$. We denote by $P_{\mu, \Sigma}$ the distribution of the random vector
\begin{align*}
    X = \mu + R \Sigma^{1/2} U,
\end{align*}
where $\mu \in \mathbb{R}^p$ and the matrix $\Sigma^{1/2} \in \mathbb{R}^{p \times p}$ is non-zero and positive semidefinite. Thus $P_{\mu, \Sigma}$ is a member of the family of non-degenerate \textit{elliptical distributions}~\citep{fang1990symmetric} with the fixed radial distribution $F_R$. This family of distributions has two parameters, $\mu$ and $\Sigma$, which can be interpreted similarly as the mean vector and the covariance matrix of the multivariate normal distribution. If $F_R$ admits a density, then the equidensity contours of $P_{\mu, \Sigma}$ are ellipses with the directions and lengths of their axes determined by the eigenvectors and eigenvalues of $\Sigma$, respectively. The following theorem confirms the intuition that $P_{\mu, \Sigma}$ should be at its most symmetric when all eigenvalues of $\Sigma$ are equal, i.e., when $P_{\mu, \Sigma}$ is \textit{spherical} and its equidensity contours spheres.
 
\begin{theorem}\label{theo:elliptical_case}
    Assume that $E(R^4) < \infty$ and denote
    \begin{align*}
        u_R = \frac{1}{2} + \frac{p - 1}{p ( \beta_R + 1 ) + 2},
    \end{align*}
    where $\beta_R = E(R^4)/ \{ E(R^2) \}^2$. Then, $(\mu, \Sigma) \mapsto O(P_{\mu, \Sigma})$ takes values in $[1/2, u_R]$ and
    \begin{itemize} 
    \item [(i)] $O(P_{\mu, \Sigma}) = 1/2$ if and only if $\Sigma$ has rank $1$.
    \item [(ii)] $O(P_{\mu, \Sigma}) = u_R$ if and only if $\Sigma = \lambda I_p$ for some $\lambda > 0$.
\end{itemize}
\end{theorem}

By Theorem \ref{theo:elliptical_case}, $O(P_{\mu, \Sigma})$ essentially measures the relative contribution of the first principal component of $P_{\mu, \Sigma}$ to its total variation: The more concentrated $P_{\mu, \Sigma}$ is on a one-dimensional subspace, the smaller the value of $O(P_{\mu, \Sigma})$. Conversely, if $P_{\mu, \Sigma}$ does not favor any particular direction, the maximal value $u_R$ is reached. These points are further illustrated in Figure \ref{fig:ellipses} where we have assumed that $p = 2$ and $R^2 \sim F(2, 8)$, making $X$ have a bivariate $t$-distribution with 8 degrees of freedom \citep{fang1990symmetric}. The values of $O(P_{\mu, \Sigma})$ given in the figure caption clearly correspond to the shape of the ellipses in the manner described in Theorem \ref{theo:elliptical_case}.

\begin{figure}
    \centering
    \includegraphics[width = 0.75\textwidth]{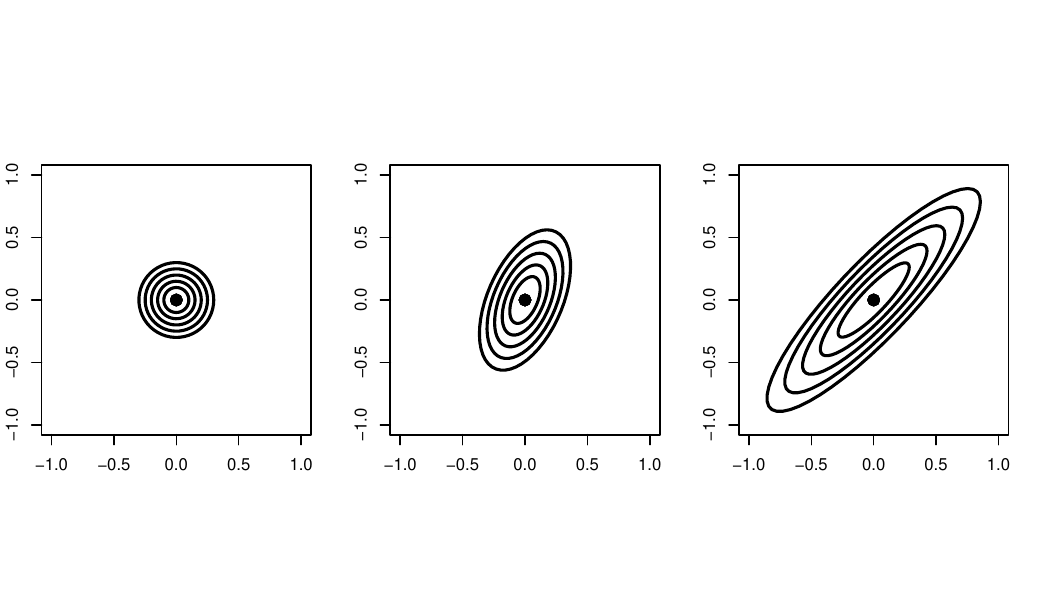}
    \caption{The equidensity contours of bivariate $t$-distribution with 8 degrees of freedom under three different covariance structures. The values of $O(P_{\mu, \Sigma})$ corresponding to the three panels are 0.600, 0.554 and 0.512, respectively.}
    \label{fig:ellipses}
\end{figure}

We additionally note the following consequences of Theorem \ref{theo:elliptical_case}: (i) When $p = 1$, the distribution $P_{\mu, \Sigma}$ equals that of its first principal component, making the two cases of Theorem \ref{theo:elliptical_case} equal, $u_R = 1/2$ regardless of $F_R$, and $O(P_{\mu, \Sigma})$ a constant function. (ii) As a function of the distribution $F_R$, the upper bound $u_R$ in Theorem~\ref{theo:elliptical_case} is maximal when $E(R^4) = \{ E(R^2) \}^2$, i.e., when $F_R$ is a Dirac point mass, in which case $u_R = p/(p + 1)$. This maximal value is reached if and only if $P_{\mu, \Sigma}$ is a uniform distribution on a sphere of arbitrary radius in $\mathbb{R}^p$, making these distributions the most symmetric, in the sense of the object shape, out of all $p$-variate elliptical distributions.

We now move to consider a hypothesis test for the null hypothesis of sphericity, $\Sigma = \lambda I_p$ for some $\lambda > 0$, giving the asymptotic null distribution of $O(P_{\mu, \lambda I_p, n})$ as our next result.

\begin{theorem}\label{theo:elliptical_case_h0}
    Let $\mu \in \mathbb{R}^p$, $\lambda > 0$ be fixed and assume that $E(R^8) < \infty$. Then,
    as $n \rightarrow \infty$,
    \begin{align*}
        \sqrt{n} \left\{ u_R - O(P_{\mu, \lambda I_p, n}) \right\} \rightsquigarrow \mathcal{N}(0, \sigma_{R}^2),
    \end{align*}
    where $u_R$ is as in Theorem \ref{theo:elliptical_case} and the constant $\sigma_R^2$ is a function of the moments of $F_R$ and is given in the proof of this theorem.
\end{theorem}

The test statistic in Theorem \ref{theo:elliptical_case_h0} is somewhat impractical in the sense that, for unknown $F_R$, its use requires estimating several higher moments of the radial variate $R$ needed to compute $u_R$ and $\sigma_R^2$. And while this is technically possible, it makes the test statistic $O(P_{\mu, \lambda I_p, n})$ less attractive in practice compared to its well-established competitors that do not necessitate this, see, e.g., \cite{hallin2006semiparametrically}.




\subsection{von Mises distribution on the circle}

Let $\mathcal{X} = \mathbb{S}^1$ be the unit circle and let $P_{\kappa}$ denote the centered von Mises distribution with the concentration parameter $\kappa \geq 0$, see, e.g., \cite{mardia2000directional}. The probability density function of $P_{\kappa}$ thus equals
\begin{align*}
    f_{\kappa}(x) = \frac{1}{2 \pi I_0(\kappa)} e^{\kappa \cos(x)}, \quad x \in [0, 2 \pi),
\end{align*}
where $I_0$ is the modified Bessel function of the first kind and order $0$. The density $f_{\kappa}$ has been plotted in Figure \ref{fig:von_mises} for various values of $\kappa$. A random variable $X \sim P_{\kappa}$ thus corresponds to a random point (its angle) on the unit sphere and the larger the value of $\kappa$ is, the more concentrated the distribution is around the zero angle. As $\kappa \rightarrow \infty$, the distribution approaches a point mass and, conversely, in the other extreme, $\kappa = 0$, we obtain the uniform distribution on the unit circle.

\begin{figure}
    \centering
    \includegraphics[width=0.75\textwidth]{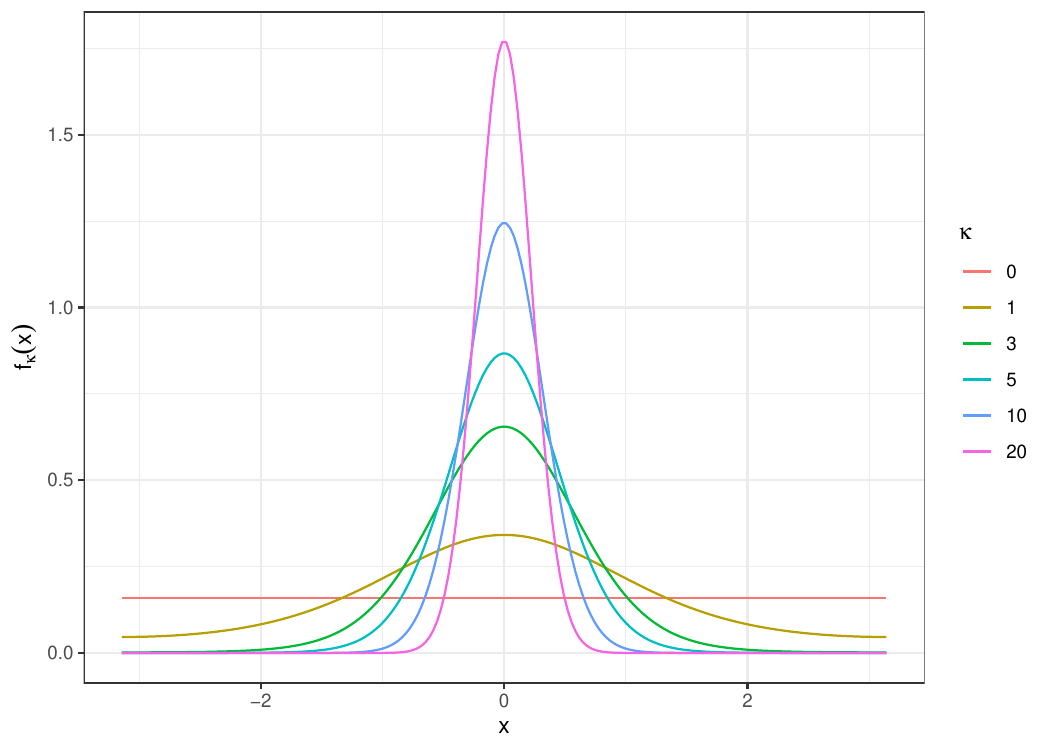}
    \caption{The density functions $f_{\kappa}$ of the von Mises distribution with various values of the concentration parameter $\kappa$. The curves show that, for $\kappa = 0$, the distribution reduces to the uniform distribution on the unit sphere, whereas when $\kappa \rightarrow \infty$, the probability mass concentrates more and more around the zero angle.}
    \label{fig:von_mises}
\end{figure}

We next equip $\mathcal{X}$ with the metric $d(x, y) = \sqrt{1 - \cos(x - y)}$, see Lemma~\ref{lem:von_mises_triangle_ineq} in the supplementary material for the proof that $d$ satisfies the triangle inequality. By the preceding discussion, it is reasonable to expect that $\kappa \mapsto O(P_{\kappa})$ is maximized at the uniform distribution, $\kappa = 0$, and the following result shows that this is indeed the case.


\begin{theorem}\label{theo:von_mises_case}
    The map $\kappa \mapsto O(P_{\kappa})$ is decreasing in $(0, \infty)$. Moreover, $O(P_{0}) = 2/3$ and $O(P_{\kappa}) \rightarrow 1/2$ as $\kappa \rightarrow \infty$.
\end{theorem}

Theorem \ref{theo:von_mises_case} reveals that (a) the least-dimensional member of the family $P_\kappa$ (in the sense of having the smallest possible $O(P_{\kappa})$) is the point mass distribution achieved in the limit $\kappa \rightarrow \infty$, and (b) the upper endpoint $2/3$ of the range of $O(P_{\kappa})$ is achieved at the uniform distribution. Hence, a test of uniformity on the unit circle can be devised based on $O(P_{\kappa, n})$, whose null distribution we derive next.

\begin{theorem}\label{theo:von_mises_case_h0}
    We have, as $n \rightarrow \infty$,
    \begin{align*}
        n \left\{ \frac{2}{3} - O(P_{0, n}) \right\} \rightsquigarrow \frac{2}{18} \chi_2^2 + \frac{1}{18} \chi_2^2,
    \end{align*}
    where the two $\chi^2_2$-variates are independent.
\end{theorem}

Up to our best knowledge, the hypothesis test corresponding to the null distribution in Theorem \ref{theo:von_mises_case_h0} appears to be novel and not equivalent to any of the classical tests of uniformity, see \cite[Section 6.3]{mardia2000directional}. This claim is further backed up
by the fact that the test statistic $O(P_{0, n})$ has a reasonably complex expression as a function of the transformed variates $(\cos (X_i), \cos^2 (X_i), \sin (X_i), \cos (X_i) \sin (X_i))$, see the proof of Theorem \ref{theo:von_mises_case_h0}. Finally, comparing the results of Theorems \ref{theo:elliptical_case_h0} and \ref{theo:von_mises_case_h0} we also observe that the convergence rate of $O(P_{n})$ to its maximum depends on the particular scenario we are in: in the former case the rate is $\sqrt{n}$ and in the latter $n$.

\subsection{Compositional data}

Fix $p > 1$ and let $\Delta^p = \{ x \in \mathbb{R}^p \mid 0 < x_1, \ldots x_p < 1, \sum_{j = 1}^p x_j = 1 \}$ denote the $p$-dimensional unit simplex. Data residing in the unit simplex are studied in compositional data analysis \citep{pawlowsky2011compositional} and occur commonly, e.g., in biology. We equip $\Delta^p$ with the Aitchison metric,
\begin{align*}
    d^2(x, y) = \frac{1}{2p} \sum_{j = 1}^p \sum_{k = 1}^p \left\{ \log \left( \frac{x_j}{x_k} \right) - \log \left( \frac{y_j}{y_k} \right) \right\}^2,
\end{align*}
which is arguably the most commonly used distance in compositional data analysis. As is typical for compositional data, where the observations sum to unity, the Aitchison distance depends on the observed vector $x$ only through the ratios of its components to one another, ignoring their absolute size.

Fix now a distribution $F_Z$ taking values on the positive real line such that $\log Z$ is symmetrically distributed around zero for $Z \sim F_Z$ . For $\mu = (\mu_1, \ldots, \mu_p)^{\mathrm{T}} \in \mathbb{R}^p$, $\mu_1, \ldots, \mu_p > 0$, and $\theta = (\theta_1, \ldots , \theta_p)^{\mathrm{T}} \in \mathbb{R}^p$, $\theta_1, \ldots, \theta_p \geq 0$, we let $P_{\mu, \theta}$ denote the distribution of the random composition
\begin{align*}
    X = (X_1, \ldots, X_p) = \left( \frac{\mu_1 Z_1^{\theta_1}}{\mu_1 Z_1^{\theta_1} + \cdots + \mu_p Z_p^{\theta_p}}, \ldots , \frac{\mu_p Z_p^{\theta_p}}{\mu_1 Z_1^{\theta_1} + \cdots + \mu_p Z_p^{\theta_p}} \right),
\end{align*}
where $Z_1, \ldots, Z_p$ are a random sample from $F_Z$. The distribution $P_{\mu, \theta}$ is similar in spirit to the Dirichlet distribution which is generated in the above manner but with the $\mu_j Z_j^{\theta_j}$ replaced by independent $\mathrm{Gamma}(\theta_j, 1)$-variates.

A simple computation reveals that, for $X \sim P_{\mu, \theta}$, the map $h: \Delta^p \to \mathbb{R}$ defined as $h(a) = E\{ d^2(X, a) \}$ is minimized uniquely at the vector $a = \mu/(\mu_1 + \cdots +\mu_p) \in \Delta^p$. Hence, $\mu$ is a location parameter and essentially the Frech\'et mean of the distribution $P_{\mu, \theta}$. The parameter $\theta$, on the other hand, controls the dispersion of $P_{\mu, \theta}$. In Figure~\ref{fig:ternary} we have illustrated the effect of $\theta$ with ternary diagrams~\citep{pawlowsky2011compositional} when $p = 3$ and $\log Z \sim \mathcal{N}(0, 1)$. When $\theta$ is a constant vector (top row), the resulting contours are symmetric w.r.t. the center of the simplex, whereas non-constant $\theta$ (bottom row) results in elongation along the axes corresponding to the largest elements of $\theta$. The corresponding values of $O(P_{\mu, \theta})$ given in the caption of Figure \ref{fig:ternary} confirm the intuitive fact that, of the four distributions, the one in the bottom left panel is the least symmetric. We also note that, as with the Dirichlet distribution, certain values of the parameter $\theta$ lead into multimodal distributions (top right panel).

\begin{figure}
    \centering
    \includegraphics[width=0.65\textwidth]{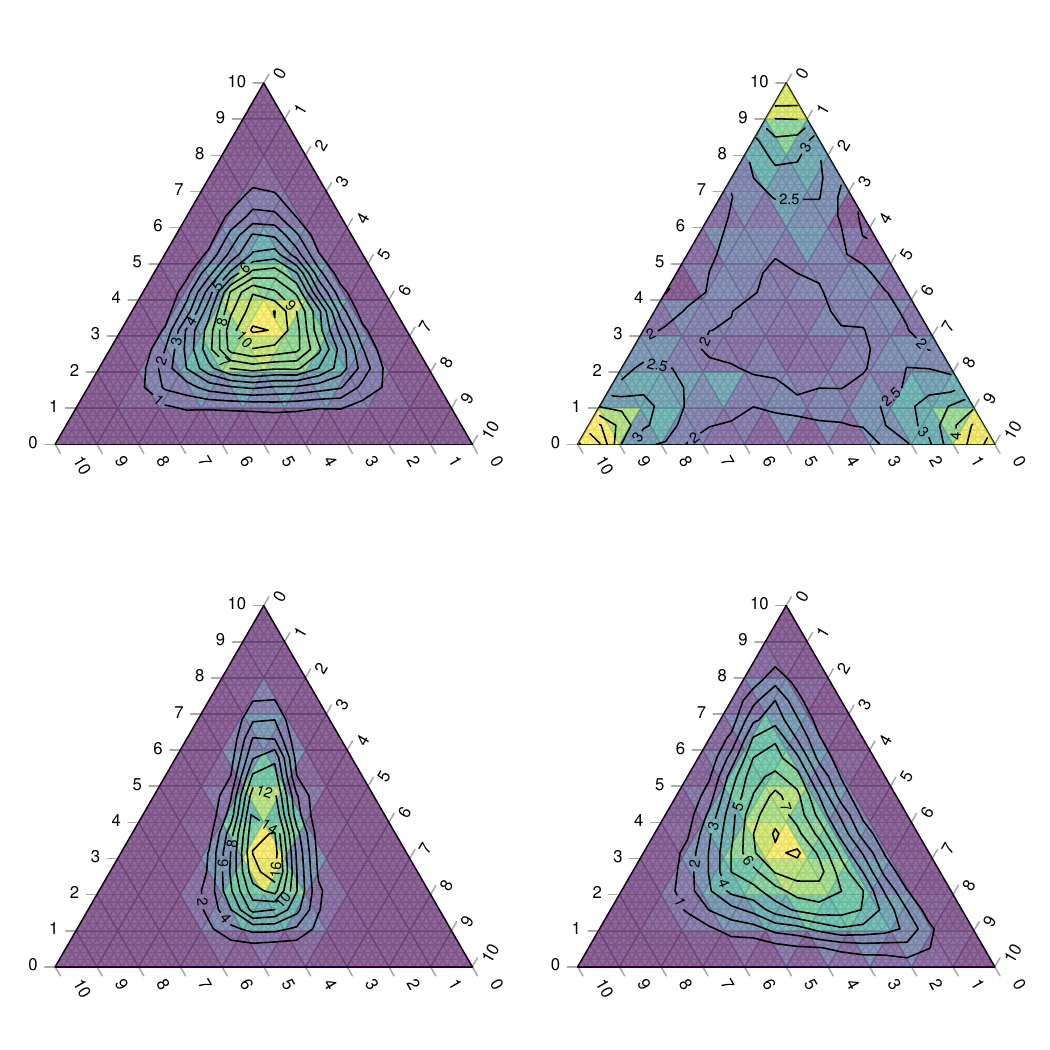}
    \caption{Density estimates of the distribution $P_{\mu, \theta}$ for $\mu = (1, 1, 1)^{\mathrm{T}}$ and four different values of $\theta$, based on samples of size $n = 10000$. The four panels correspond, from left to right, top to bottom, to the values $\theta = (0.5, 0.5, 0.5), (1.25, 1.25, 1.25), (0.75, 0.25, 0.25), (0.75, 0.75, 0.25)$. The measure $O(P_{\mu, \theta})$ takes values in $[0.5, 0.625]$ and in the four scenarios its population value is 0.625, 0.625, 0.547, 0.594, respectively. The plots were drawn using the \texttt{R}-package \texttt{Ternary} \citep{RSmith2017}. }
    \label{fig:ternary}
\end{figure}

The following result shows that $O(P_{\mu, \theta})$ is invariant to $\mu$ and essentially measures how close to a constant vector the dispersion parameter $\theta$ is.

\begin{theorem}\label{theo:compositional_case}
    Assume that $E\{ (\log Z)^4 \} < \infty$ and denote
    \begin{align*}
        u_Z = \frac{1}{2} + \frac{p(p - 2)}{(p - 1)(\gamma_Z + 2p + 1) + 4},
    \end{align*}
    where $\gamma_Z = E\{ (\log Z)^4 \}/[E\{ (\log Z)^2 \} ]^2$. Then, $(\mu, \theta) \mapsto O(P_{\mu, \theta})$ takes values in $[1/2, u_Z]$ and
    \begin{itemize} 
    \item [(i)] $O(P_{\mu, \theta}) = 1/2$ if and only if exactly one of $\theta_1, \ldots, \theta_p$ is non-zero.
    \item [(ii)] $O(P_{\mu, \theta}) = u_Z$ if and only if $\theta_1 = \cdots = \theta_p$.
\end{itemize}
\end{theorem}

The implications of Theorem \ref{theo:compositional_case} include; (i) $O(P_{\mu, \theta})$ achieves its maximal value when the components of the composition are exchangeable, leading to the symmetry observed in the top panels of Figure \ref{fig:ternary}. (ii) $O(P_{\mu, \theta}) = 1/2$ only when $X$ is essentially one-dimensional. In terms of ternary plots, such a distribution would appear as a one-dimensional curve on the simplex. (iii) Comparison of Theorems \ref{theo:elliptical_case} and \ref{theo:compositional_case} shows that their upper bounds and equality conditions share a certain resemblance. This is not a coincidence as the Aitchison distance between compositions $x, y \in \Delta^p$ is equal to the Euclidean distance between the isometric logratio (ilr) transformations of $x$ and $y$.

We next derive a test for the null hypothesis of symmetry, $\theta_1 = \cdots = \theta_p$. For simplicity, we compute the limiting null distribution only in the special case where $F_Z$ is the standard log-normal distribution, $\log(Z) \sim \mathcal{N}(0, 1)$. Equivalent results for other distributions $F_Z$ could be derived using the same proof techniques, by computing the relevant moments up to the eighth order, see the proof of Theorem \ref{theo:compositional_case_h0}.

\begin{theorem}\label{theo:compositional_case_h0}
    Let $\mu = (\mu_1, \ldots, \mu_p)^{\mathrm{T}} \in \mathbb{R}^p$, $\mu_1, \ldots, \mu_p > 0$, and $\theta_0 > 0$ be fixed, denote $1_p = (1, \ldots , 1)^{\mathrm{T}} \in \mathbb{R}^p$ and assume that $\log (Z) \sim \mathcal{N}(0, 1)$. Then,
    as $n \rightarrow \infty$,
    \begin{align*}
        \sqrt{n} \left\{ \frac{2p - 1}{2p + 2} - O(P_{\mu, \theta_0 1_p, n}) \right\} \rightsquigarrow  \mathcal{N} \left( 0, \frac{(p - 2)^2}{2 (p + 1)^3 (p - 1)} \right).
    \end{align*}
\end{theorem}

As with the circular data earlier, also the above test appears to be novel. However, its practical usefulness is limited by its parametric nature.

\subsection{Discrete metric in a finite space}

Fix $p \geq 3$, and let $\Theta^p$ denote the set of all vectors $\theta = (\theta_1, \ldots, \theta_p)^{\mathrm{T}} \in [0, 1]^p$ with elements summing to one and having at least two non-zero elements. For $\theta \in \Theta^p$, we denote by $P_\theta$ the discrete distribution in $\{1, \ldots, p\}$ taking the value $i$ with the probability $\theta_i$, $i = 1, \ldots, p$. Thus, $\Theta^p$ indexes the set of all non-degenerate probability distributions on a $p$-element set.

We equip the support set $\{1, \ldots, p \}$ with the discrete metric, $d(i, j) = 1 - \mathbb{I}(i = j)$, where $\mathbb{I}(\cdot)$ denotes the indicator function. While this metric is extremely simple, interestingly, it still leads to a meaningful and useful characterization of shape among the distributions $P_\theta$ via the quantity $O(P_\theta)$, as evidenced by the following theorem.

\begin{theorem}\label{theo:discrete_case}
    The map $\theta \mapsto O(P_{\theta})$ takes values in $[1/2, 1 - 1/p]$ and
    \begin{itemize} 
    \item [(i)] $O(P_{\theta}) = 1/2$ if and only if $\theta$ has exactly two non-zero elements.
    \item [(ii)] $O(P_{\theta}) = 1 - 1/p$ if and only if $\theta = (1/p) 1_p$.
\end{itemize}
\end{theorem}

By Theorem \ref{theo:discrete_case}, $O(P_{\theta})$ measures the uniformity of $P_\theta$. The maximal value $1 - 1/p$ is reached precisely when each of the $p$ objects has exactly the same probability mass. Now, given a sample distribution $P_n$ and the associated observed relative frequencies $(n_1/n, \ldots , n_p/n)$, a classical way of testing the null hypothesis of uniformity is via Pearson's chi squared statistic,
\begin{align*}
    T_n = \sum_{i = 1}^p n p \left( \frac{n_i}{n} - \frac{1}{p} \right)^2,
\end{align*}
which satisfies $T_n \rightsquigarrow \chi^2_{p - 1}$ when $n \rightarrow \infty$. Our next result shows that the sample object shape is actually asymptotically equivalent to $T_n$ under the null hypothesis that $\theta = (1/p) 1_p$. Hence, in this simple case, the object shape recovers the optimal test of symmetry.

\begin{theorem}\label{theo:discrete_case_h0}
    For $\theta_0 = (1/p) 1_p$, we have, as $n \rightarrow \infty$,
    \begin{align*}
        \frac{p(p - 1)}{p - 2} n \left\{ 1 - \frac{1}{p} - O(P_{\theta_0, n}) \right\} = T_n + o_p(1) \rightsquigarrow \chi^2_{p - 1}.
    \end{align*}
\end{theorem}

\section{Peeling plot}\label{sec:peeling}

Our proposed tool of object data visualization, the peeling plot, admits two variants, maximization and minimization, and we next introduce them one-by-one.

Let $X_1, \ldots, X_n$ be an observed object sample in some metric space $(\mathcal{X}, d)$. We construct the maximization peeling plot of the sample as follows. Letting $P_{n, -i}$ denote the empirical distribution of the sample with the $i$th observation removed, we first identify the index $i = 1, \ldots, n$ for which $O(P_{n, -i})$ is maximized. Denoting this index by $j_1$, we then remove the $j_1$th observation from the sample and iteratively repeat this process. In the end, we obtain the vector of indices $(j_1, j_2, \ldots, j_n)$, giving the order in which the observations were peeled (removed), and the vector $o \in [1/2, 1]^n$ containing the successive values of the object shape produced by this process. Thus, in particular, $o_1 = O(P_{n, -j_1})$. The peeling plot is then the scatter plot of $o$ versus $1, \ldots, n$, with the accompanying index values $(j_1, j_2, \ldots, j_n)$. For discrete sample space $\mathcal{X}$, it might happen that at some stage of the peeling process we are left with a sample of identical objects, making the object shape undefined. To avoid such situations, we found it useful to carry out the process only until $80\%$ of the total sample size $n$ has been peeled. \texttt{R}-code for computing the peeling plot is available on the author's web page https://users.utu.fi/jomivi/software/.

\begin{figure}
    \centering
    \includegraphics[width=0.50\textwidth]{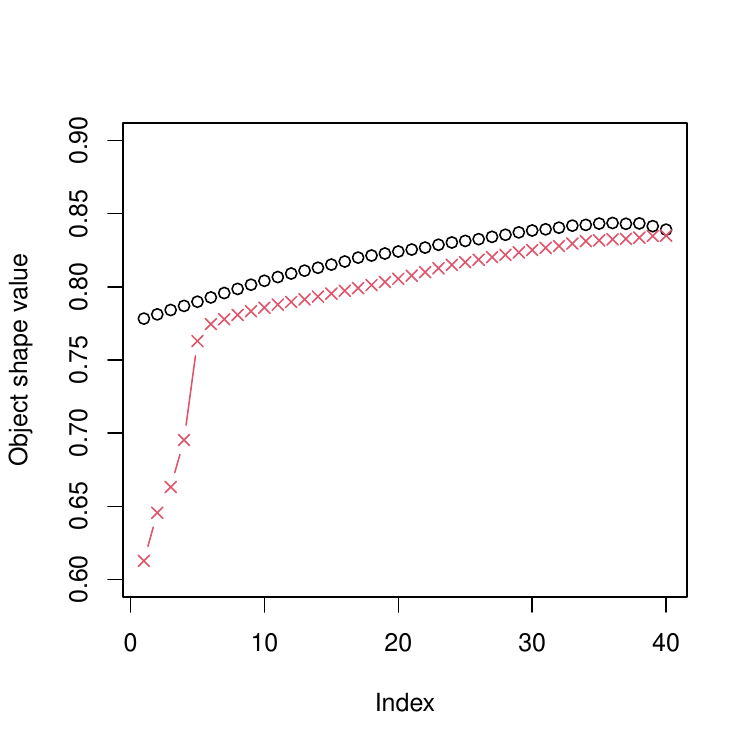}
    \caption{Examples of two superimposed maximization peeling plots. The red crosses correspond to a sample having 5 outliers, the effect of their peeling being clearly visible as a steep increase in the plot.}
    \label{fig:peeling_1}
\end{figure}

Based on Theorem \ref{theo:geometric_characterization}, this process attempts to find, in a greedy fashion, subsamples that are essentially as symmetric as possible, in the sense of Theorem~\ref{theo:geometric_characterization} and the metric $d$. While the finer interpretation of the plot is, in general, dependent on the particular setting at hand, we found that it is very useful in detecting outliers regardless of the metric space. This is because (a) having isolated outliers makes satisfying the condition in part (ii) of Theorem \ref{theo:geometric_characterization} very unlikely, and (b) the effect of outliers on $O(P)$ is rather strong due to the presence of the fourth moments of the distances in \eqref{eq:main_concept}, implying that they get removed first in the peeling. We have demonstrated this in Figure~\ref{fig:peeling_1}, which shows the peeling plots of a particular sample of size $n = 50$ without outliers (black circles) and another sample of size $n = 50$ with five outliers (red crosses). The main difference between the two curves is the sharp increase in the red curve, which represents the sudden rise in $O(P)$ as soon as all five outliers have been peeled from the data. 

We further investigated the outlier detection capabilities of the peeling plot in a simulation study. We took $(\mathcal{X}, d)$ to be the space of all positive definite $3 \times 3$ matrices equipped with the affine invariant metric \citep{bhatia2009positive}. We generated the observations as $X_i = U_i \mathrm{diag}\{\exp(z_{i1}), \exp(z_{i2}), \exp(z_{i3}) \} U_i^{\mathrm{T}}$ where the $3 \times 3$ orthogonal matrix $U_i$ is drawn uniformly w.r.t. the Haar measure and $z_{i1}, z_{i2}, z_{i3} \sim \mathcal{N}(\theta, 1)$, independently. The bulk of the data was generated using the value $\theta = 0$, whereas a small proportion $\varepsilon = 0.05, 0.10$ was taken to be outliers, generated either with $\theta = 2$ (less separated case) or $\theta = 4$ (more separated case). We considered the sample sizes $n = 20, 40, 60, 80, 100$ and in each replicate of the simulation attempted to detect the outlying observations using the following four methods: (1) The peeling plot where all observations preceding the largest jump in the plot are taken as outliers. (2) The ``oracle'' peeling plot where we label the first $n \varepsilon$ peeled observations as outliers. (3) Another oracle-type estimator which computes metric lens depths (a measure of non-outlyingness) of the sample as described in \cite{cholaquidis2020weighted, geenens2023statistical} and labels the $n \varepsilon$ objects with the smallest depths as outliers. (4) Similar as the previous method, but instead of knowing the exact amount of outliers, we order the metric lens depths of the sample in increasing order and label all points left of the single largest jump as outliers. Each of the four methods produces an index set of observations it labeled as outliers, and as our final evaluation criterion we use the F1-scores (harmonic mean of precision and recall) between these and the true set of outlier indices.

\begin{figure}
    \centering
    \includegraphics[width=0.75\textwidth]{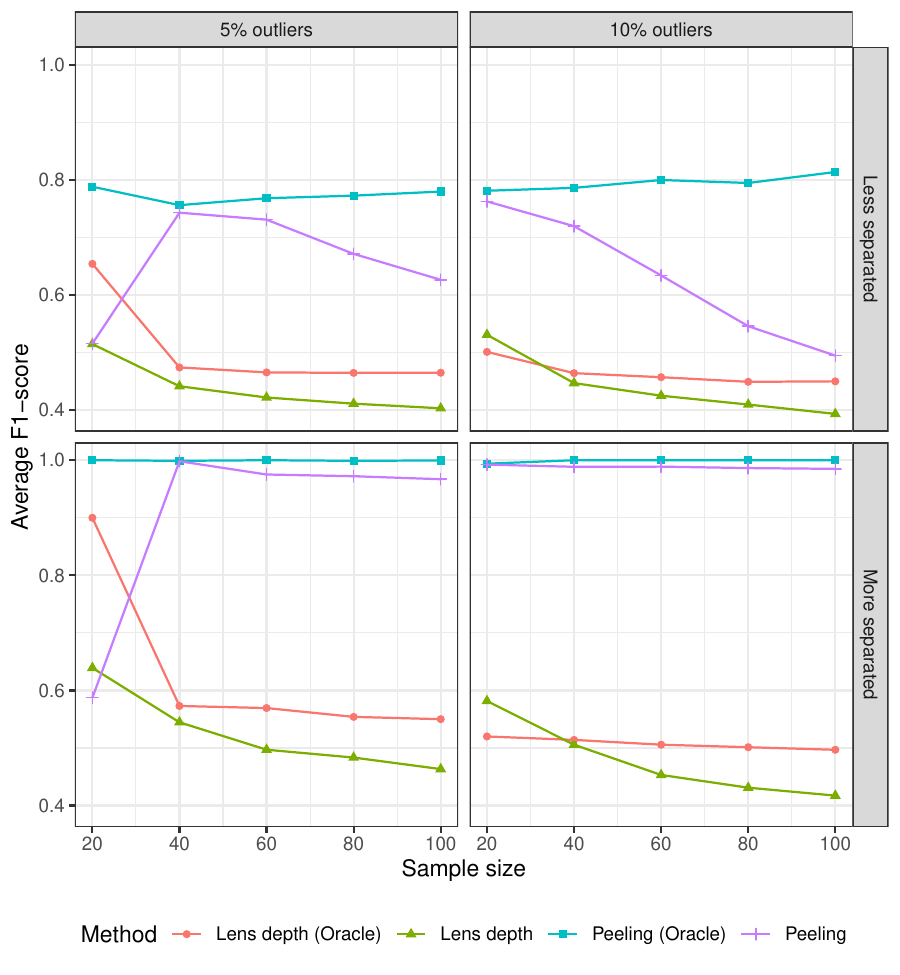}
    \caption{Average F1-scores in the outlier detection simulation, grouped according to sample size, proportion of outliers, separation of the groups and method.}
    \label{fig:outlier_simu}
\end{figure}

The average F1-scores over 500 replicates of the simulation are shown in Figure \ref{fig:outlier_simu}. The main implications of the plots are: (i) The oracle peeling plot gives the best performance, obtaining almost perfect F1-score when the outliers are sufficiently far away from the bulk (``More separated'' case). (ii) The lens depth, with or without oracle knowledge, gives subpar F1-scores. Closer inspection of the results reveals that the lens depth does identify many of the true outliers, but its F1-score is lowered by having large amounts of false positives interspersed with the actual outliers. The reason for this is likely the fact that lens depth is, unlike the object shape, a \textit{robust} statistical measure and as such not sensitive to the effects of any single observation, making it, in turn, less effective in identifying outliers. (iii) The F1-score of the peeling plot without oracle knowledge (purple curve with plus-signs) deteriorates with growing sample size $n$. This non-intuitive phenomenon is caused by the fact that our criterion for selecting the outlier set (location of single largest jump) is very ``local''. That is, having a larger sample size makes it more likely that the true outlier set is masked by the bulk purely by chance. Better alternatives could likely be devised with, e.g., bootstrapping strategies that control for the randomness in the bulk. Nevertheless, in scenarios with well-separated bulk and outliers, even the current heuristic criterion appears to work extremely well.

\begin{figure}
    \centering
    \includegraphics[width=0.60\textwidth]{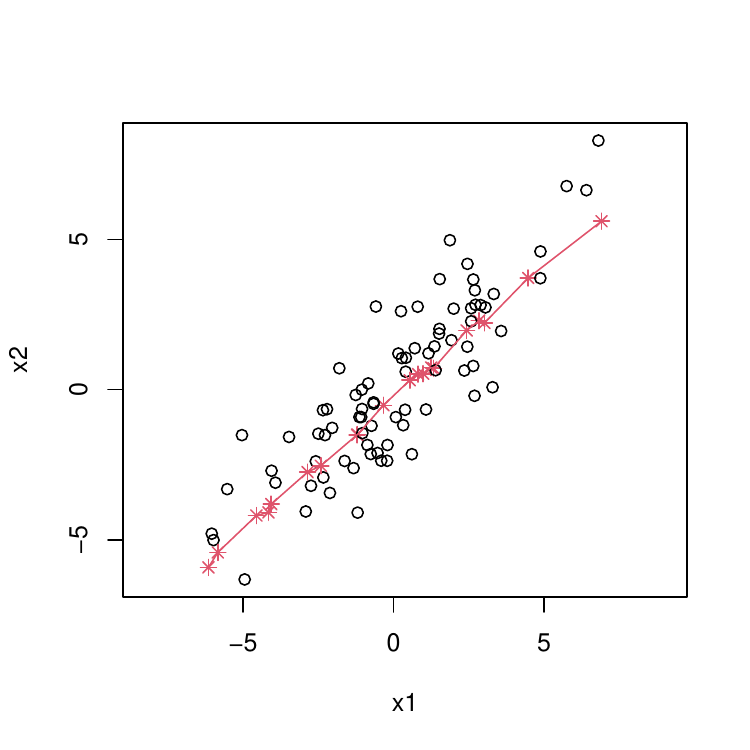}
    \caption{A simple bivariate dataset where the $n_0 = 20$ observations recovered by minimization peeling (red stars) recover the first principal component direction in the data.}
    \label{fig:pc1}
\end{figure}

We next turn our attention to the minimization peeling plot which is constructed in exact analogy to its maximization counterpart, by simply minimizing (instead of maximizing) the object shape at every peeling. By Theorem \ref{theo:geometric_characterization}, this process aims to find a subsample which is as concentrated onto a line (in the metric $d$) as possible. In analogy to classical PCA, if we stop the peeling when $n - n_0$ observations are peeled, the remaining $n_0$ objects thus form a set of representatives of the first principal component ``direction'' of the sample. We have demonstrated this in the case of Euclidean data in Figure \ref{fig:pc1}, where the $n_0 = 20$ red crosses which were left unpeeled in a sample of size $n = 100$ clearly capture the leading principal direction. In the plot we have further connected the crosses in an order which (approximately) minimizes the ``surplus'', $| d(x_1, x_2) + \cdots + d(x_{n_0-1}, x_{n_0}) - d(x_1, x_{n_0})|$, in the triangle inequality, using the implementation of the traveling salesman problem in the R-package \texttt{TSP} \citep{RTSP}.

\begin{figure}
    \centering
    \includegraphics[width=0.6\textwidth]{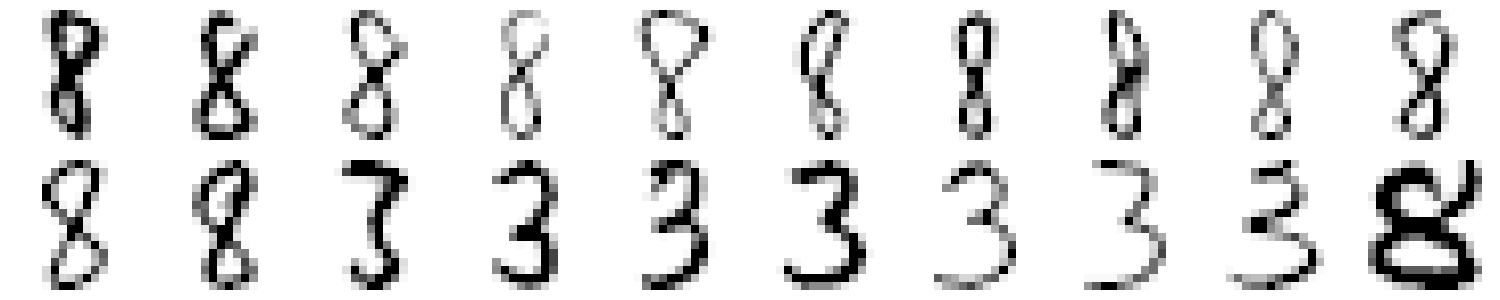}
    \caption{From left to right, top to bottom, the $n_0 = 20$ representatives of the first principal ``direction'' in a dataset of size $n = 100$ consisting of greyscale images of digits 3 and 8.}
    \label{fig:pc2}
\end{figure}

While the interpretation in the Euclidean setting in  Figure \ref{fig:pc1} is clear, in general metric spaces the situation can be more elaborate. For example, on a unit circle, the resulting line can ``join'' with itself by traversing around the circle. Nevertheless, the resulting set of $n_0$ observations can in any case be interpreted as an ordered sequence of objects corresponding to the single largest ``axis'' of variation in the data. To further demonstrate this, we applied minimization peeling to the hand-written digit data set available in the R-package \texttt{tensorBSS} \citep{RtensorBSS}. Each observation in the data is a greyscale image of size $16 \times 16$ and of these we took a random subsample of size $n = 100$ of digits 3 and 8 only. The $n_0 = 20$ representatives of the first principal direction extracted from this sample using minimization peeling with the Manhattan distance are shown in Figure \ref{fig:pc2}, ordered from left to right, top to bottom. The sequence of objects is dominated by the group structure separating the two digit classes. Within each group we further have a continuum corresponding to a smooth change in digit shape. For example, the threes transition from more thickly drawn digits into thinner ones. Finally, at the very end we have a single outlying digit 8 drawn differently from the rest, making it resemble more a three than an eight and lie on the other side of the threes group.

\section{Discussion}

A natural continuation of this work would be to investigate how efficient the hypothesis tests obtained using object shape are compared to their parametrically optimal counterparts. It is clear that, as a ``payment'' for the generality of our proposed concept, the obtained procedures cannot be expected to be optimal, but how much efficiency exactly is lost should be studied. As remarked in Section \ref{sec:cases}, in the discrete case we indeed recover the optimal test, but this is likely an artifact of the extreme simplicity of the scenario.

A question which we have ignored thus far is whether the extremal value $O(P) = 1$ can actually be reached in practice. This is indeed possible under very extreme scenarios. Take, e.g., a uniform distribution on the unit circle in $\mathbb{R}^2$ equipped with the railroad metric (all paths between points on distinct origin-centered rays go through the origin). Then, $d(X_1, X_2) = d(X_2, X_3) = d(X_3, X_1) = 1$ almost surely, giving $O(P) = 1$ by part (ii) of Theorem~\ref{theo:geometric_characterization}.


\section*{Acknowledgement}
This work was supported by the Academy of Finland (grants no. 347501, 353769). The author wishes to thank H. Saarinen for a discussion regarding modified Bessel functions of the first kind.

\appendix

\section{Proofs of theoretical results}\label{sec:proofs}

\begin{proof}[Proof of Theorem \ref{theo:geometric_characterization}]
For part (i), denoting $d_{ij} = d(X_i, X_j)$, we use the reverse triangle inequality to conclude that
\begin{align}\label{eq:reverse_1}
    d_{23}^2 \geq | d_{12} - d_{13} |^2,
\end{align}
where equality is reached if and only if either $d_{12} = d_{23} + d_{31}$ or $d_{13} = d_{12} + d_{23}$. Rearranging \eqref{eq:reverse_1} we get,
\begin{align}\label{eq:reverse_2}
    2 d_{12} d_{13} \geq d_{12}^2 + d_{13}^2 - d_{23}^2,
\end{align}
Now, the triangle inequality gives $d_{12}^2 + d_{13}^2 - d_{23}^2 \geq - 2 d_{12} d_{13}$, implying that squaring \eqref{eq:reverse_2} and taking expectation gives $3 E (d_{12}^4) \leq 6 E ( d_{12}^2 d_{13}^2 )$, where equality is reached if and only if,
\begin{align*}
    2 d_{12} d_{13} = | d_{12}^2 + d_{13}^2 - d_{23}^2 |
\end{align*}
almost surely. The case of a positive sign leads to the reverse triangle inequality \eqref{eq:reverse_2} and the case of a negative sign occurs if and only if $(d_{12} + d_{13})^2 = d_{23}^2$, i.e., $d_{12} + d_{13} = d_{23}$. Thus equality is reached in $E (d_{12}^4) \leq 2 E (d_{12}^2 d_{13}^2)$ if and only if the three points reside almost surely on a line.

For part (ii), the power mean inequality gives,
\begin{align*}
    \frac{d_{12}^4 + d_{23}^4 + d_{31}^4}{3} \geq \left( \frac{d_{12}^2 + d_{23}^2 + d_{31}^2}{3} \right)^2,
\end{align*}
where equality is reached if and only if $d_{12} = d_{23} = d_{31}$. Taking the expectation shows that,
\begin{align*}
    E ( d_{12}^4 ) \geq \frac{1}{9} \{ 3 E ( d_{12}^4 ) + 6 E ( d_{12}^2 d_{13}^2 ) \},
\end{align*}
where equality is reached if and only if $d_{12} = d_{23} = d_{31}$ almost surely. This proves part (ii).

\end{proof}

\begin{lemma}\label{lem:moments_sphere}
Let $U$ have uniform distribution on the unit sphere $\mathbb{S}^{p - 1}$. Then,
\begin{align*}
    E(U \otimes U) &= \frac{1}{\sqrt{p}} v_p \\
    E(U U^{\mathrm{T}} \otimes U U^{\mathrm{T}}) &= \frac{1}{p(p + 2)}(I_{p^2} + K_{p,p} + p v_p v_p^{\mathrm{T}}),
\end{align*}
where $K_{p,p}$ is the $(p, p)$-commutation matrix and $v_p = \mathrm{vec}(I_p)/\sqrt{p}$.
\end{lemma}

\begin{proof}[Proof of Lemma \ref{lem:moments_sphere}]
The first result follows directly from \cite[Theorem 2.7]{fang1990symmetric}. For the second one, the symmetry and exchangeability of the elements of $U$ implies that the only non-zero elements of $E(UU^{\mathrm{T}} \otimes UU^{\mathrm{T}})$ are those on the diagonal and the $(i, j)$th and $(j, i)$th elements of each $(i, j)$th off-diagonal block, $i \neq j$, of $E(UU^{\mathrm{T}} \otimes UU^{\mathrm{T}})$. The diagonal elements corresponding to quadruply repeated indices all equal $E(U_1^4) = 3/\{ p ( p + 2) \}$, while the remaining diagonal elements and the off-diagonal elements all equal $E(U_1^2 U_2^2) = 1/\{ p ( p + 2) \}$ \citep[Theorem 3.3]{fang1990symmetric}. Arranging the elements to a matrix now gives the claim.

\end{proof}

\begin{proof}[Proof of Theorem \ref{theo:elliptical_case}]
    First, we may without loss of generality take $\mu = 0$, since the Euclidean metric depends on the observations only through their pairwise differences.
    
    We next derive an expression for the quantity $O(P_{\mu, \Sigma})$. For this, Theorem 3.3 in \cite{fang1990symmetric} gives that
    \begin{align}\label{eq:uniform_moments_1}
        E(u_1^2 u_2^2) = \frac{1}{p(p + 2)} \quad \mbox{and} \quad E(u_1^4) = \frac{3}{p(p + 2)}.
    \end{align}
    These two expressions can be used to show, with some algebra, that
    \begin{align}\label{eq:uniform_moments_2}
        E \{ (u^{\mathrm{T}} \Sigma u)^2 \} = \frac{1}{p(p + 2)} \{ \mathrm{tr} (\Sigma) \}^2 + \frac{2}{p(p + 2)} \mathrm{tr}(\Sigma^2).
    \end{align}
    Next, using \eqref{eq:uniform_moments_1} and \eqref{eq:uniform_moments_2}, it is straightforwardly seen that,
    \begin{align}\label{eq:elliptical_g}
        O(P_{\mu, \Sigma}) = \frac{E \|X_1 - X_2\|^2 \|X_1 - X_3\|^2 }{E( \|X_1 - X_2\|^4 )} = \frac{(\gamma \beta_R + 3) + 2 \gamma \beta_R G}{(2 \gamma \beta_R + 2) + (4 \gamma \beta_R + 4) G},
    \end{align}
    where $\gamma = p/(p + 2)$, $G = \mathrm{tr} (\Sigma^2)/\{ \mathrm{tr} (\Sigma) \}^2$ and $\beta_R = E(R^4)/\{ E(R^2) \}^2$. By the power mean inequality, we have that $G \geq 1/p$, with equality if and only if all eigenvalues of $\Sigma$ are equal, i.e., $\Sigma = \lambda I_p$ for some $\lambda > 0$ (the scenario $\lambda = 0$ is not allowed as we ruled out the case $\Sigma = 0$).

    Differentiation reveals that the mapping $x \mapsto (A + Bx)/(C + Dx)$, where $A, B, C, D > 0$, is strictly increasing in $[0, \infty)$ if $AD > BC$. Now, since $\beta_R \geq 1$, this condition applies to the expression on the right-hand side of \eqref{eq:elliptical_g} and the earlier inequality $G \geq 1/p$ guarantees that,
    \begin{align*}
        O(P_X) \geq \frac{(\gamma \beta_R + 3) + 2 \gamma \beta_R/p}{(2 \gamma \beta_R + 2) + (4 \gamma \beta_R + 4)/p},
    \end{align*}
    with equality if and only if $\Sigma = \lambda I_p$ for some $\lambda > 0$. Simplifying now yields part (ii) of the claim. For part (i), it is sufficient to use the same argument with the inequality $G \leq 1$, where equality is reached if and only if all but the first eigenvalue of $\Sigma$ are zero. 
\end{proof}


\begin{proof}[Proof of Theorem \ref{theo:elliptical_case_h0}]
The sample estimator $O(P_{\mu, \lambda I_p, n})$ can be written as
\begin{align*}
    O(P_{\mu, \lambda I_p, n}) = \frac{\frac{1}{n^3} \sum_{i = 1}^n \sum_{j = 1}^n \sum_{k = 1}^n d(X_i - \bar{X}, X_j - \bar{X})^2 d(X_i - \bar{X}, X_k - \bar{X})^2}{\frac{1}{n^2} \sum_{i = 1}^n \sum_{j = 1}^n \sum_{k = 1}^n d(X_i - \bar{X}, X_j - \bar{X})^4}.
\end{align*}
Simplifying this further shows that,
\begin{align}\label{eq:elliptical_sample_expansion}
    O(P_{\mu, \lambda I_p, n}) = \frac{\tilde{m}_n + 3 \{ \mathrm{tr}(\tilde{S}_n) \}^2}{2 \tilde{m}_n + 4 \mathrm{tr}(\tilde{S}_n^2) + 2 \{ \mathrm{tr}(\tilde{S}_n) \}^2} = \frac{\tilde{m}_n + 3 p v_p^{\mathrm{T}} \tilde{s}_n \tilde{s}_n^{\mathrm{T}} v_p}{2 \tilde{m}_n + 4 \tilde{s}_n^{\mathrm{T}} \tilde{s}_n + 2 p v_p^{\mathrm{T}} \tilde{s}_n \tilde{s}_n^{\mathrm{T}} v_p},
\end{align}
where $\tilde{m}_n = (1/n) \sum_{i = 1}^n \| X_i - \bar{X} \|^4$, $\tilde{S}_n = (1/n) \sum_{i = 1}^n (X_i - \bar{X}) (X_i - \bar{X})^{\mathrm{T}}$, $\tilde{s}_n = \mathrm{vec}(\tilde{S}_n)$ and $v_p = \mathrm{vec}(I_p)/\sqrt{p}$.

Take next $\mu = 0$, $\lambda = 1$, without loss of generality, and denote $q_r = E(R^r)$. Then $E(UU^{\mathrm{T}}) = (1/p) I_p$, meaning that $\tilde{S}_n \rightarrow_p (q_2/p) I_p$. Thus, by the central limit theorem and Slutsky's theorem, $\sqrt{n} \{ \tilde{S}_n -  (q_2/p) I_p \} = \sqrt{n} \{ S_n -  (q_2/p) I_p \} + o_p(1) $ where $S_n = (1/n) \sum_{i = 1}^n X_i X_i^{\mathrm{T}}$. Similarly, $\tilde{m}_n \rightarrow_p q_4 $, and $\sqrt{n} ( \tilde{m}_n -  q_4 ) = \sqrt{n} ( m_n -  q_4 ) + o_p(1) $ where $m_n = (1/n) \sum_{i = 1}^n \| X_i \|^4$.

For us to use the delta method, we require the joint limiting distribution of $(\mathrm{vec}(S_n), m_n)$. The previous expansions, along with the central limit theorem, give this to be
\begin{align*}
    \sqrt{n} \begin{pmatrix}
        \mathrm{vec}(S_n) - (q_2/\sqrt{p}) v_p \\
        m_n - q_4
    \end{pmatrix} \rightsquigarrow \mathcal{N}_{p^2 + 1} \left( 0, \begin{pmatrix}
        \mathrm{cov}(X \otimes X) & \mathrm{Cov}( \| X \|^4, X \otimes X) \\
        \mathrm{cov}(X \otimes X, \| X \|^4 ) & \mathrm{var}(\| X \|^4)
    \end{pmatrix} \right).
\end{align*}
By Lemma \ref{lem:moments_sphere} and the unit length of $U$, the covariance matrix above takes the form,
\begin{align*}
    \Theta = \begin{pmatrix}
        \frac{q_4}{p(p + 2)}(I_{p^2} + K_{p,p} + p v_p v_p^{\mathrm{T}}) - \frac{q_2^2}{p} v_p v_p^{\mathrm{T}}  & (q_6 - q_4 q_2) \frac{1}{\sqrt{p}} v_p \\
        (q_6 - q_4 q_2) \frac{1}{\sqrt{p}} v_p^{\mathrm{T}} & q_8 - q_4^2
    \end{pmatrix},
\end{align*}
where $K_{p,p}$ is the $(p, p)$-commutation matrix. Define next $h: \mathbb{R}^{p^2 + 1} \to \mathbb{R}$ to be the map
\begin{align*}
    h(s, m) = \frac{m + 3 p v_p^{\mathrm{T}} s s^{\mathrm{T}} v_p}{2 m + 4 s^{\mathrm{T}} s + 2 p v_p^{\mathrm{T}} s s^{\mathrm{T}} v_p},
\end{align*}
whose gradient is seen to evaluate to
\begin{align*}
    \nabla h((q_2/\sqrt{p}) v_p, q_4) = \frac{p (p - 1) q_2}{\{ p (q_4 + q_2^2) + 2 q_2^2 \}^2}
    \begin{pmatrix}
    2 \sqrt{p} q_4 v_p\\
    - q_2
    \end{pmatrix}.
\end{align*}
After some simplification, the resulting asymptotic variance is seen to be,
\begin{align*}
    \sigma_R^2 =& \nabla h((q_2/\sqrt{p}) v_p, q_4)^{\mathrm{T}} \Theta \nabla h((q_2/\sqrt{p}) v_p, q_4) \\
    =& \frac{p^2 (p - 1)^2 q_2^2}{\{ p (q_4 + q_2^2) + 2 q_2^2 \}^4} \{ 4 q_4^2 (q_4 - q_2^2) - 4 q_4 q_2 (q_6 - q_4 q_2) + q_2^2 (q_8 - q_4^2) \},
\end{align*}
concluding the proof.
\end{proof}

\begin{lemma}\label{lem:von_mises_triangle_ineq}
Let $\alpha, \beta, \gamma \in \mathbb{R}$ be arbitrary. Then,
    \begin{align*}
        \sqrt{1 - \cos(\alpha - \gamma)} \leq \sqrt{1 - \cos(\alpha - \beta)} + \sqrt{1 - \cos(\beta - \gamma)}.
    \end{align*}
\end{lemma}

\begin{proof}[Proof of Lemma \ref{lem:von_mises_triangle_ineq}]
The result follows by applying the triangle inequality for the Euclidean metric in $\mathbb{R}^2$ to the points $(\cos \alpha, \sin \alpha)$, $(\cos \beta, \sin \beta)$, $(\cos \gamma, \sin \gamma)$, and using the identity $\cos(\alpha - \beta) = \cos \alpha \cos \beta + \sin \alpha \sin \beta$.
\end{proof}

\begin{proof}[Proof of Theorem \ref{theo:von_mises_case}]
    To simplify \eqref{eq:main_concept} we use the standard multiple angle and power reduction formulas for trigonometric functions, along with the identities \citep[Section 3.5.4]{mardia2000directional},
    \begin{align*}
        E \{ \cos(mX) \} = \frac{I_m(\kappa)}{I_0(\kappa)} \quad \mbox{and} \quad E \{ \sin(mX) \} = 0,
    \end{align*}
    valid for all $m \in \mathbb{N}$, where $X \sim P_{\kappa}$ and $I_m(\kappa)$ denotes the modified Bessel function of the first kind and order $m$. These yield,
    \begin{align}\label{eq:g_p_kappa}
        O(P_{\kappa}) = \frac{2 I_0^3 - 3 I_0 I_1^2 + I_1^2 I_2 }{3 I_0^3 - 4 I_0 I_1^2 + I_0 I_2^2},
    \end{align}
    where $I_m \equiv I_m(\kappa)$.
    
    Using next the notation $x = I_1/I_0$ and $y = I_2/I_0$, we obtain
    \begin{align}\label{eq:factor_1}
        O(P_{\kappa}) - \frac{1}{2} = \left(\frac{1 - y}{2}\right) \left( \frac{1 - 2x^2 + y}{3 - 4 x^2 + y^2} \right).
    \end{align}
    By \cite[Lemma 3]{robert1990modified} and \citep[Appendix A]{watson1983statistics}, $I_1/I_0$ and $I_2/I_1$ are increasing functions of $\kappa$, taking values in $[0, 1]$. Consequently $1 - y = 1 - (I_2/I_1)(I_1/I_0)$ is  decreasing in $\kappa$ and takes values in $[0, 1]$. 

    The desired monotonicity of $O(P_\kappa)$ now follows once we show that
    \begin{align}\label{eq:final_map}
        \kappa \mapsto h(\kappa) := \frac{I_0^2 - 2I_1^2 + I_0 I_2}{3 I_0^2 - 4 I_1^2 + I_2^2}
    \end{align}
    is decreasing. From \cite[Lemma 2.1]{yang2018monotonicity} we have the everywhere convergent power series representations
    \begin{align*}
        I_0^2(\kappa) =& \sum_{n = 0}^\infty \frac{(2n)!}{n! n! n! n!} \left( \frac{\kappa^2}{2} \right)^{n}, \\
        I_1^2(\kappa) =& \sum_{n = 1}^\infty \frac{(2n)!}{(n - 1)! n! n! (n + 1)!} \left( \frac{\kappa^2}{2} \right)^{n}, \\
        I_0(\kappa) I_2(\kappa) =& \sum_{n = 1}^\infty \frac{(2n)!}{(n - 1)! (n - 1)! (n + 1)! (n + 1)!} \left( \frac{\kappa^2}{2} \right)^{n},  \\
        I_2^2(\kappa) =& \sum_{n = 2}^\infty \frac{(2n)!}{(n - 2)! n! n! (n + 2)!} \left( \frac{\kappa^2}{2} \right)^{n}.
    \end{align*}
    Consequently, the numerator of \eqref{eq:final_map} has the power series $\sum_{n = 0}^\infty a_n (2n)! (\kappa^2/2)^{n}$, where
    \begin{align*}
        a_0 = 1, \quad a_1 = \frac{1}{4},
    \end{align*}
    and
    \begin{align*}
        a_n = \frac{1}{(n!)^4} \left( 1 - \frac{2n}{n + 1} + \frac{n^2}{(n + 1)^2} \right) = \frac{1}{(n!)^4 (n + 1)^2},
    \end{align*}
    for $n \geq 2$. For the denominator of \eqref{eq:final_map}, we analogously have the power series $\sum_{n = 0}^\infty b_n (2n)! (\kappa^2/2)^{n}$, where
    \begin{align*}
        b_0 = 3, \quad b_1 = 1,
    \end{align*}
    and
    \begin{align*}
        b_n = \frac{1}{(n!)^4} \left( 3 - \frac{4n}{n + 1} + \frac{(n - 1)n}{(n + 1)(n + 2)} \right) = \frac{6}{(n!)^4 (n + 1) (n + 2)},
    \end{align*}
    for $n \geq 2$. Consequently, the ratios $a_n/b_n$ equal
    \begin{align*}
        \frac{a_0}{b_0} = \frac{1}{3}, \quad \frac{a_1}{b_1} = \frac{1}{4}, \quad \frac{a_n}{b_n} = \frac{n + 2}{6(n + 1)}.
    \end{align*}
    We then see that the sequence $a_n/b_n$ is decreasing, implying by \cite[Lemma 2.2]{yang2018monotonicity} that so is the map $\kappa \mapsto h(\sqrt{\kappa})$. Consequently, also the original function in \eqref{eq:final_map} is decreasing, proving the decreasingness of $O(P_\kappa)$.

    What remains now is to establish the values of $O(P_0)$ and $O(P_\infty)$. We have $I_0(0) = 1$ and $I_m(0) = 0$ for $m \geq 1$. Consequently, plugging in to \eqref{eq:g_p_kappa} shows that $O(P_0) = 2/3$. For the case $\kappa \rightarrow \infty$, we use the asymptotic expansions in \citep[Appendix 1]{mardia2000directional} to obtain,
    \begin{align*}
        \frac{I_m(\kappa)}{I_0(\kappa)} = 1 + \frac{m^2}{2 \kappa} + \frac{m^2(m^2 - 2)}{8 \kappa^2} + \mathcal{O}\left( \frac{1}{\kappa^3} \right).
    \end{align*}
    Plugging these in into \eqref{eq:g_p_kappa}, we get
    \begin{align*}
        O(P_{\kappa}) = \frac{\frac{3}{\kappa^2} + \mathcal{O}\left( \frac{1}{\kappa^3} \right)}{\frac{6}{\kappa^2} + \mathcal{O}\left( \frac{1}{\kappa^3} \right)} \rightarrow \frac{1}{2},
    \end{align*}
    as $\kappa \rightarrow \infty$, completing the proof.
\end{proof}

\begin{proof}[Proof of Theorem \ref{theo:von_mises_case_h0}]
    Using again standard multiple angle and power reduction formulas for trigonometric functions, the quantity $O(P_{0, n})$ in \eqref{eq:main_concept_sample} takes the form
    \begin{align*}
        O(P_{0, n}) = \frac{1 - 2a^2 - c^2 + 2ace + a^2 b - bc^2}{2(1 - a^2 - c^2 + b^2 - b + e^2)} =: h(a, b, c, e),
    \end{align*}
    where $(a, b, c, e) = \frac{1}{n} \sum_{i = 1}^n (\cos (X_i), \cos^2 (X_i), \sin (X_i), \cos (X_i) \sin (X_i))^{\mathrm{T}}$. By the central limit theorem, the vector $(a, b, c, e)$ has the limiting distribution,
    \begin{align*}
        \sqrt{n} \{ (a, b, c, e)^{\mathrm{T}} - (0, 1/2, 0, 0)^{\mathrm{T}} \} \rightsquigarrow \mathcal{N}_4(0, \mathrm{diag}(1/2, 1/8, 1/2, 1/8)),
    \end{align*}
    as $n \rightarrow \infty$. It is straightforward (if tedious), to check that the gradient of $(a, b, c, e) \mapsto h(a, b, c, e)$ is zero at $(0, 1/2, 0, 0)$, meaning that we need to resort to second-order delta method, which says that the limiting distribution of $n \{ O(P_{0, n}) - 2/3 \}$ is the same as the distribution of $(1/2) Z^{\mathrm{T}} \Sigma^{1/2} H \Sigma^{1/2} Z$, where $H \in \mathbb{R}^{4 \times 4}$ is the Hessian of the map $(a, b, c, e) \mapsto h(a, b, c, e)$ at $(0, 1/2, 0, 0)$, $\Sigma = \mathrm{diag}(1/2, 1/8, 1/2, 1/8)$ and $Z \sim \mathcal{N}_4(0, I_4)$. After some differentiation, we obtain that $H = \mathrm{diag}(-2/9, -16/9, -2/9, -16/9)$, which finishes the proof.
\end{proof}

\begin{proof}[Proof of Theorem \ref{theo:compositional_case}]
    The Aitchison metric $d^2(X_1, X_2)$ for $X_1, X_2 \sim P_{\mu, \theta}$ is invariant to the choice of $\mu$, meaning that we can take $\mu_j \equiv 1$, without loss of generality. We begin by deriving a simplified expression for the denominator $E \{ d(X_1, X_2)^4 \}$, where $X_1$ and $X_2$ are i.i.d. from the distribution $P_\theta$ with the respective ``generating variates'' $Z_1, \ldots , Z_p$ and $W_1, \ldots , W_p$.  
    For this, we denote $A_i = \log Z_i$, $B_i = \log W_i$, $C_i = \theta_i A_i$ and $D_i = \theta_i B_i$. Moreover, we use the notation that $S_q = \sum_{i = 1}^p C_i^q$, $T_q = \sum_{i = 1}^p D_i^q$ and $M = \sum_{i = 1}^p C_i D_i$. Consequently,
    \begin{align*}
        4 p^2 E \{ d(X_1, X_2)^4 \} =&  E \left( \left[ \sum_{i, j} \{ (C_i - C_j)^2 - 2 (C_i - C_j) (D_i - D_j) + (D_i - D_j)^2 \} \right]^2 \right) \\
        =& E [ \{ 2(p S_2 - S_1^2) - 4 (p M - S_1 T_1) + 2(p T_2 - T_1^2) \}^2 ].
    \end{align*}
    We denote $u_q = E(A_1^q)$, $\psi_q = \sum_{i = 1}^p \theta_i$ and observe that, by the symmetry of $\log Z$, we have $u_q = 0$ for all off $q$. We thus obtain the following moments,
    \begin{align*}
        E(S_2^2) = E(S_2 S_1^2) =& \psi_4 (u_4 - u_2^2) + \psi_2^2 u_2^2 \\
        E(S_1^4) =& \psi_4 (u_4 - 3 u_2^2) + 3 \psi_2^2 u_2^2 \\
        E \{ (p S_2 - S_1^2)^2 \} =& (p - 1)^2 \{ \psi_4 (u_4 - u_2^2) + \psi_2^2 u_2^2 \} - 2 (\psi_4 - \psi_2^2) u_2^2 \\
        E(M^2) = E(M S_1 T_1) =& \psi_4 u_2^2 \\
        E(S_1^2) = E(S_2) =& \psi_2 u_2 \\
        E\{ (p M - S_1 T_1)^2 \} =& (p - 1)^2 \psi_4 u_2^2 - (\psi_4 - \psi_2^2) u_2^2 \\
        E\{ (p S_2 - S_1^2) (p T_2 - T_1^2) \} =& (p - 1)^2 \psi_2^2 u_2^2 \\
        E\{ (p S_2 - S_1^2) (p M - S_1 T_1) \} =& 0.
    \end{align*}
    Using these, we get that
    \begin{align*}
        4 p^2 E \{ d(X_1, X_2)^4 \} =& 8 (p - 1)^2 (\psi_4 u_4 + \psi_4 u_2^2 + 2 \psi_2^2 u_2^2 ) - 32 (\psi_4 - \psi_2^2) u_2^2 \\
        =& u_2^2 \psi_2^2 \{ 8 (p - 1)^2 (\beta \gamma_Z + \beta + 2) - 32 (\beta - 1) \},
    \end{align*}
    where $\beta = \psi_4/\psi_2^2$ and $\gamma_Z = u_4/u_2^2$. Using a similar technique, we also obtain,
    \begin{align*}
        & 4 p^2 E \{ d(X_1, X_2)^2 d(X_1, X_3)^2 \}\\
        =& 4 (p - 1)^2 (\psi_4 u_4 - \psi_4 u_2^2 + 4 \psi_2^2 u_2^2) - 8 (\psi_4 - \psi_2^2) u_2^2 \\
        =&  u_2^2 \psi_2^2 \{ 4 (p - 1)^2 (\beta \gamma_Z - \beta + 4) - 8 (\beta - 1) \}.
    \end{align*}
    Putting everything together, we get
    \begin{align*}
        O(P_{\mu , \theta}) =& \frac{ 4 (p - 1)^2 (\beta \gamma_Z - \beta + 4) - 8 (\beta - 1) }{  8 (p - 1)^2 (\beta \gamma_Z + \beta + 2) - 32 (\beta - 1) }.
    \end{align*}
    The remainder of the proof is similar to the conclusion of the proof of Theorem \ref{theo:elliptical_case}. That is, we first check that the map $\beta \mapsto O(P_{\mu, \theta})$ is strictly decreasing whenever $p > 1$, after which the desired result follows by using the bounds $1/p \leq \beta \leq 1$ and their equality conditions.

\end{proof}


\begin{proof}[Proof of Theorem \ref{theo:compositional_case_h0}]
    We may, without loss of generality, take $\theta_0 = 1$, $\mu_j \equiv 1$. We begin by deriving an asymptotic expansion for
    \begin{align*}
        U_n =& \frac{1}{n^3} \sum_{i = 1}^n \sum_{j = 1}^n \sum_{k = 1}^n d(X_i, X_j)^2 d(X_i, X_k)^2 \\
        =& \frac{1}{n^3} \sum_{i,j,k} \sum_{s,t,u,v} (a_{is} - a_{it} - a_{js} + a_{jt})^2(a_{iu} - a_{iv} - a_{ku} + a_{kv})^2,
    \end{align*}
    where $a_{is}$, $i = 1, \ldots , n$, $s = 1, \ldots , p$, form a random sample from the distribution of $\log Z_1$. By standard asymptotic arguments and using the symmetry of $\log Z_1$, we get,
    \begin{align*}
        U_n =&   \frac{1}{n} \sum_i \sum_{s,t,u,v} (a_{is} - a_{it})^2 (a_{iu} - a_{iv})^2 + 3 \left\{ \frac{1}{n} \sum_i \sum_{s,t} (a_{is} - a_{it})^2 \right\}^2 + o_p\left( \frac{1}{\sqrt{n}} \right) \\
        =:& B_{1n} + 3 B_{2n}^2 + o_p\left( \frac{1}{\sqrt{n}} \right).
    \end{align*}
    Similarly, we obtain,
    \begin{align*}
        V_n =& \frac{1}{n^3} \sum_{i = 1}^n \sum_{j = 1}^n \sum_{k = 1}^n d(X_i, X_j)^4 \\
        =& 2 B_{1n} + 2 B_{2n}^2 + 4 \sum_{s,t,u,v} \left\{ \frac{1}{n} \sum_i (a_{is} - a_{it}) (a_{iu} - a_{iv}) \right\}^2 + o_p\left( \frac{1}{\sqrt{n}} \right).
    \end{align*}
    Expanding reveals that the third term above can be written as
    \begin{align*}
        16 p^2 \mathrm{tr}(S_n^2) + 16 (1_p^{\mathrm{T}} S_n 1_p)^2 - 32 p 1_p^{\mathrm{T}} S_n^2 1_p,
    \end{align*}
    where $S_n = (1/n) \sum_{i = 1}^n a_i a_i^{\mathrm{T}}$ and $a_i = (a_{i1}, \ldots , a_{ip})^{\mathrm{T}}$. Moreover, the quantities $B_{1n}$, $B_{2n}$ can be written as,
    \begin{align*}
        B_{1n} =& \frac{4 p^2}{n} \sum_{i = 1}^n a_i Q a_i a_i^{\mathrm{T}} Q a_i =: 4 p^2 m_n \\
        B_{2n} =& 2 p \mathrm{tr}(S_n) - 2 \cdot 1_p^{\mathrm{T}} S_n 1_p,
    \end{align*}
    where $Q = I_p - (1/p) 1_p 1_p^{\mathrm{T}} $. Hence, $O(P_{\mu, \theta_0 1_p, n})$ is equal to 
    \begin{align*}
        \frac{p^2 m_n + 3 p^2 \{ \mathrm{tr}(S_n) \}^2 - 6 p \mathrm{tr}(S_n) 1_p^{\mathrm{T}} S_n 1_p + 3 (1_p^{\mathrm{T}} S_n 1_p)^2}{2 p^2 m_n + 2 p^2 \{ \mathrm{tr}(S_n) \}^2 - 4 p \mathrm{tr}(S_n) 1_p^{\mathrm{T}} S_n 1_p + 4 p^2 \mathrm{tr}(S_n^2) + 6 (1_p^{\mathrm{T}} S_n 1_p)^2 - 8 p 1_p^{\mathrm{T}} S_n^2 1_p},
    \end{align*}
    up to an $o_p(1/\sqrt{n})$-quantity. By the independence of the elements $a_{ij}$ and using \cite[Section 6.2.2]{petersen2008matrix} we may compute
    \begin{align*}
        E(m_n) =& (p - 1)(p + 1) u_2^2 + (1/p) (p - 1)^2 (u_4 - 3 u_2^2) \\
        E(S_n) =& u_2 I_p,
    \end{align*}
    where $u_q = E(a_{11}^q)$. Under the assumption of log-normality for $Z$, we have that $a_{ij}$ form a random sample from $\mathcal{N}(0, 1)$ and $u_2 = 1$, $u_4 = 3$. Moreover, denoting $v_p = (1/\sqrt{p}) \mathrm{vec}(I_p)$, by the central limit theorem,
    \begin{align*}
        & \sqrt{n} \begin{pmatrix}
            \mathrm{vec}(S_n) - \sqrt{p} v_p \\
            m_n - (p - 1)(p + 1)          
        \end{pmatrix} \\
        \rightsquigarrow& 
        \mathcal{N}_{p^2 + 1} \left( 0, \begin{pmatrix}
        \mathrm{cov}(A \otimes A) & \mathrm{cov}\{  (A^{\mathrm{T}} Q A)^2, A \otimes A \} \\
        \mathrm{cov}\{ A \otimes A, (A^{\mathrm{T}} Q A)^2 \} & \mathrm{var}\{ (A^{\mathrm{T}} Q A)^2 \}
    \end{pmatrix} \right),
    \end{align*}
    where $A \sim \mathcal{N}_p(0, I_p)$. From the proof of Theorem \ref{theo:elliptical_case_h0} we obtain that $\mathrm{cov}(A \otimes A) = I_{p^2} + K_{p,p}$ where $K_{p,p}$ is the $(p, p)$ commutation matrix. For the remaining variance terms, we denote $y = QA, z = PA$ where $P = I_p - Q = (1/p) 1_p 1_p^{\mathrm{T}}$, making $y, z$ jointly normally distributed and independent. Consequently, $y^{\mathrm{T}}y \sim \chi^2_{p - 1}$ and
    \begin{align*}
    E ( A A^{\mathrm{T}} Q A A^{\mathrm{T}} Q A A^{\mathrm{T}} ) =& E(y y^{\mathrm{T}} y y^{\mathrm{T}} y y^{\mathrm{T}}) + (p - 1)(p + 1) P\\
    =& U E(w w^{\mathrm{T}} w w^{\mathrm{T}} w w^{\mathrm{T}}) U^{\mathrm{T}} + (p - 1)(p + 1) P,
    \end{align*}
    where $U \in \mathbb{R}^{p \times (p - 1)}$ has orthonormal columns and satisfies $UU^{\mathrm{T}} = Q$, and $w \sim \mathcal{N}_{p - 1}(0, I_{p - 1})$. A brute force computation shows that $E(w w^{\mathrm{T}} w w^{\mathrm{T}} w w^{\mathrm{T}})$ is diagonal matrix whose diagonal elements each equal $(p + 1)(p + 3)$. Consequently,
    \begin{align*}
         & \mathrm{cov}\{  (A^{\mathrm{T}} Q A)^2, A \otimes A \} \\
         =& (p + 1)(p + 3) \mathrm{vec}(Q) + \frac{1}{p} (p - 1)(p + 1) 1_{p^2} - (p - 1)(p + 1) \mathrm{vec}(I_p) \\
         =& \frac{4 (p + 1) }{p} (p \sqrt{p} v_p - 1_{p^2})
    \end{align*}
    Using the same orthogonal decomposition we also obtain that
    \begin{align*}
        E ( A^{\mathrm{T}} Q A A^{\mathrm{T}} Q A A^{\mathrm{T}} Q A A^{\mathrm{T}} Q A ) = (p - 1)(p + 1)(p + 3)(p + 5),
    \end{align*}
    giving
    \begin{align*}
        \mathrm{var}\{ (A^{\mathrm{T}} Q A)^2 \} =& (p - 1)(p + 1)(p + 3)(p + 5) - (p - 1)^2(p + 1)^2 \\
        =& 8 (p - 1)(p + 1)(p + 2)
    \end{align*}
    We denote by $h: \mathbb{R}^{p^2 + 1} \to \mathbb{R}$ the map taking $(s, m)$ to
    \begin{align*}
        \frac{p^2 m + 3 p^3 v_p^{\mathrm{T}} s s^{\mathrm{T}} v_p - 6 p \sqrt{p} v_p^{\mathrm{T}} s s^{\mathrm{T}} 1_{p^2}  + 3 1_{p^2}^{\mathrm{T}} s s^{\mathrm{T}} 1_{p^2}}{2 p^2 m + 2 p^3 v_p^{\mathrm{T}} s s^{\mathrm{T}} v_p - 4 p \sqrt{p} v_p^{\mathrm{T}} s s^{\mathrm{T}} 1_{p^2} + 4 p^2 s^{\mathrm{T}}s + 6 1_{p^2}^{\mathrm{T}} s s^{\mathrm{T}} 1_{p^2} - 8 p s^{\mathrm{T}} (1_p 1_p^{\mathrm{T}} \otimes I_p) s}.
    \end{align*}
    Letting $t(s, m)$ and $b(s, m)$ denote the numerator and denominator, respectively, of this fraction, we have
    \begin{align*}
        t(\sqrt{p} v_p, (p - 1)(p + 1)) =& 2 p^2 (2p - 1)(p - 1) \\
        b(\sqrt{p} v_p, (p - 1)(p + 1)) =& 4 p^2 (p + 1)(p - 1)
    \end{align*}
    we get the gradient with respect to $s$
    \begin{align*}
        \nabla_s h(s, m) = \frac{1}{b(s, m)^2} & \left[
             b(s, m) \{ 6p^3 v_p v_p^{\mathrm{T}} s - 6 p \sqrt{p} (v_p 1^{\mathrm{T}}_{p^2} + 1_{p^2} v_p^{\mathrm{T}} )s + 6 \cdot 1_{p^2} 1^{\mathrm{T}}_{p^2} s\} \right. \\
             & 
             - t(s, m) \{ 4p^3 v_p v_p^{\mathrm{T}} s - 4 p \sqrt{p} (v_p 1^{\mathrm{T}}_{p^2} + 1_{p^2} v_p^{\mathrm{T}} )s + 8p^2 s \\
             & \left. 
             + 12 \cdot 1_{p^2} 1^{\mathrm{T}}_{p^2} s - 16 p (1_p 1_p^{\mathrm{T}} \otimes I_p) s \}
                \right].        \end{align*}
    Plugging in, this reduces to
    \begin{align*}
        & \nabla_s h(\sqrt{p} v_p, (p - 1)(p + 1)) \\
        =& \frac{24 p^3 (p + 1)(p - 1)^2 (p \sqrt{p} v_p - 1_{p^2}) - 8 p^3 (2p - 1)(p - 1) (p + 1) ( p \sqrt{p} v_p - 1_{p^2} ) }{\{ 4 p^2 (p + 1)(p - 1)\}^2} \\
        =& \frac{8 p^3 (p + 1)(p - 1)(p - 2) (p \sqrt{p} v_p - 1_{p^2}) }{\{ 4 p^2 (p + 1)(p - 1)\}^2} \\
        =& \frac{ (p - 2) (p \sqrt{p} v_p - 1_{p^2}) }{ 2 p (p + 1) (p - 1)}.
    \end{align*}
    whereas the gradient with respect to $m$ is
    \begin{align*}
        \nabla_m h(s, m) = \frac{p^2 b(s, m) - 2 p^2 t(s, m) }{b(s, m)^2},
    \end{align*}
    which evaluates to
    \begin{align*}
        \nabla_m h(\sqrt{p} v_p, (p - 1)(p + 1)) = \frac{- (p - 2) }{ 4 (p + 1)^2 (p - 1) }.
    \end{align*}
    Delta-method now gives the variance of the estimator to be
    \begin{align*}
        \frac{ (p - 2)^2 }{ 2 (p + 1)^2 (p - 1)} - \frac{(p - 2)^2}{(p + 1)^2 (p - 1)} + \frac{(p - 2)^2 (p + 2)}{2 (p + 1)^3 (p - 1)},
    \end{align*}
    concluding the proof.
\end{proof}

\begin{proof}[Proof of Theorem \ref{theo:discrete_case}]
    We first note that $O(P_{\theta})$ admits the simpler form
    \begin{align*}
        O(P_{\theta}) = \frac{1 - 2 s_2 + s_3}{1 - s_2} = \frac{\sum_{i = 1}^p \theta_i (1 - \theta_i)^2}{\sum_{i = 1}^p \theta_i (1 - \theta_i)},
    \end{align*}
    where $s_k = \sum_{i = 1}^p \theta_i^k$.

    For claim (i), assume that $\theta \in \Theta^p$ is such that $O(P_{\theta}) = 1/2$ and $\theta$ has at least three non-zero elements. Then, it is simple to check that there is a non-zero probability that equality is not reached in the triangle inequality $1 - \mathbb{I}(X_i = X_k) = 2 - \mathbb{I}(X_i = X_j) - \mathbb{I}(X_j = X_k)$ for any permutation $(i, j, k)$ of the elements $(1, 2, 3)$ (this happens if $X_i, X_j, X_k$ all take a different value from each other). By Theorem \ref{theo:geometric_characterization} this leads to a contradiction and, hence, $O(P_{\theta}) = 1/2$ implies that $\theta$ must have exactly two non-zero elements (since we disallowed the case of a single non-zero element).

    To show the converse, assume that $\theta$ has exactly two non-zero elements. Without loss of generality, we can take $\theta = (q, 1 - q, 0, \ldots , 0)$. Then,
    \begin{align*}
        O(P_{\theta}) = \frac{1 - 2q^2 - 2(1 - q)^2 + q^3 + (1 - q)^3}{1 - q^2 - (1 - q)^2} = \frac{1}{2},
    \end{align*}
    proving the claim (i).

    For the claim (ii), we write
    \begin{align*}
        O(P_{\theta}) = 1 - \frac{s_2 - s_3}{1 - s_2},
    \end{align*}
    after which an application of Lemma \ref{lem:probability_lemma} in Appendix \ref{sec:proofs} yields the desired claim.
\end{proof}


\begin{lemma}\label{lem:probability_lemma}
    Let $p \geq 3$ and let $\theta \in \Theta^p$. Then,
    \begin{align*}
        \sum_{k = 1}^p \theta_k (1 - \theta_k) \leq p \sum_{k = 1}^p \theta_k^2 (1 - \theta_k),
    \end{align*}
    with equality if and only if $\theta_k = 1/p$ for all $k \in \{ 1, \ldots , p \}$.
\end{lemma}

\begin{proof}[Proof of Lemma \ref{lem:probability_lemma}]
    We divide the proof in two disjoint cases. Consider first the case where $\theta_k \leq 1 - 1/p$ for all $k$ and let $f: \mathbb{R} \to \mathbb{R}$ denote the map $f(x) = x(1 - x)$. Then, for a fixed $k = 1, \ldots, p$,
    \begin{align*}
        & (\theta_k - 1/p)\left\{ f(\theta_k) - \frac{1}{p} \sum_{\ell = 1}^p f(\theta_\ell) \right\}\\
        =& (\theta_k - 1/p)\left\{ f(\theta_k) - f(1/p) \right\} + (\theta_k - 1/p)\left\{ f(1/p) - \frac{1}{p} \sum_{\ell = 1}^p f(\theta_\ell) \right\} \\
        \geq&  (\theta_k - 1/p)\left\{ f(1/p) - \frac{1}{p} \sum_{\ell = 1}^p f(\theta_\ell) \right\},
    \end{align*}
    where the inequality holds because $f$ achieves its maximum in $[0, 1/p]$ at $p = 1/p$ and its minimum in $[1/p, 1 - 1/p]$ at the end points. Also, equality in the above inequality is reached if and only if $\theta_k \in \{ 1/p, 1 - 1/p \}$. Taking now the mean of the previous inequality over $k$, we obtain the desired inequality. As $p \geq 3$, the condition for equality follows from the previous individual sufficient and necessary conditions. 

    Assume next that $\theta_p =: q > 1 - 1/p$ (in which case all other $\theta_k$ are strictly smaller than $1/p$) and recall that $p \geq 3$. Now $\sum_{k = 1}^{p - 1} \theta_k = 1 - q$ and the Cauchy-Schwarz inequality gives that $\sum_{k = 1}^{p - 1} \theta_k^2 \geq (1 - q)^2/(p - 1)$ with equality if and only if $\theta_k = (1 - q)/(p - 1)$ for all $k = 1, \ldots , p - 1$. Consequently,
    \begin{align}\label{eq:inequality_chain_K}
    \begin{split}
        & \sum_{k = 1}^p \theta_k (1 - \theta_k) - p \sum_{k = 1}^p \theta_k^2 (1 - \theta_k) \\
        \leq&  1 - q - \frac{1}{p - 1}(1 - q)^2 + q ( 1 - q ) - p \sum_{k = 1}^{p - 1} \theta_k^2 (1 - \theta_k) - p q^2 (1 - q) \\
        =& 1 - q^2 - \frac{1}{p - 1}(1 - q)^2 - p \sum_{k = 1}^{p - 1} \theta_k^2 (1 - \theta_k) - p q^2 (1 - q). 
    \end{split}
    \end{align}
    Now, Cauchy-Schwarz gives us,
    \begin{align*}
        (p - 2 + q) \sum_{k = 1}^{p - 1} \theta_k^2 (1 - \theta_k) \geq \left( \sum_{k = 1}^{p - 1} \theta_k (1 - \theta_k) \right)^2,
    \end{align*}
    and we furthermore have that $\sum_{k = 1}^{p - 1} \theta_k^2 \leq (1 - q)^2$ where equality is reached if and only if exactly one of $\theta_k$, $k = 1, \ldots, p - 1$, is non-zero. This gives,
    \begin{align*}
        (p - 2 + q) \sum_{k = 1}^{p - 1} \theta_k^2 (1 - \theta_k) \geq q^2 (1 - q)^2,
    \end{align*}
    which, when plugged in into \eqref{eq:inequality_chain_K}, gives,
    \begin{align}\label{eq:inequality_chain_K_2}
    \begin{split}
        & \sum_{k = 1}^p \theta_k (1 - \theta_k) - p \sum_{k = 1}^p \theta_k^2 (1 - \theta_k) \\
        & \leq 1 - q^2 - \frac{1}{p - 1}(1 - q)^2 - \frac{p}{p - 2 + q} q^2 (1 - q)^2 - p q^2 (1 - q) \\
        & \leq 1 - q^2 - \frac{1}{p - 1}(1 - q)^2 - \frac{p}{p - 1} \left( 1 - \frac{1}{p} \right) q (1 - q)^2 - p q^2 (1 - q) \\
        &= \frac{1 - q}{p - 1}\left\{  (p - 1)(1 + q) - (1 - q) - (p - 1) q ( 1 - q ) - p (p - 1) q^2 \right\} \\
        &= \frac{1 - q}{p - 1}\left\{ p - 2 + q - (p - 1)^2 q^2 \right\},
     \end{split}
    \end{align}
    where equality cannot hold in both of the inequalities at the same time (one of them requires equal $\theta_k$, $k = 1, \ldots, p - 1$, and one of them requires exactly one of them to be non-zero).
    
    Taking now the expression in brackets on the right-hand side of \eqref{eq:inequality_chain_K_2} and treating it as a function $g$ of $q > 1 - 1/p$, we have,
    \begin{align*}
        g'(q) = 1 - 2 (p - 1)^2 q < 1 - \frac{2 (p - 1)^3}{p},
    \end{align*}
    which is strictly negative for $p \geq 3$. Finally,
    \begin{align*}
        O(1 - 1/p) =& \frac{1}{p^2} \{ p^2 (p - 2) + p (p - 1) - (p - 1)^4 \}\\
        =& -\frac{1}{p^2} \{ 1 + p (p - 1)^2 (p - 3) \},
    \end{align*}
    which is strictly negative for $p \geq 3$. These two results about $g$ yield that $g(q) < 0$ for all $q > 1 - 1/p$ and $p \geq 3$, showing, by the right-hand side of \eqref{eq:inequality_chain_K_2}, that the desired inequality holds (strictly) in this case, finishing the proof.
\end{proof}

\begin{proof}[Proof of Theorem \ref{theo:discrete_case_h0}]
    Let $s_k = \sum_{i = 1}^p (n_i/n)^k$. We then note that Pearson's chi squared statistic can be expressed as $n(p s_2 - 1)$. Plugging this is, we observe that the claim of this theorem is equivalent to having
    \begin{align*}
        n \frac{(p + 1) s_2 - 1 - p(p - 1) s_3 + p(p - 2) s_2^2}{(p - 2)(1 - s_2)} = o_p(1).
    \end{align*}
    Thus, the proof is complete once we show that
    \begin{align}\label{eq:s2_s3_connection}
        n\{ (p + 1) s_2 - 1 - p(p - 1) s_3 + p(p - 2) s_2^2 \} = o_p(1).
    \end{align}
    Letting $x = (n_1/n, \ldots, n_p/n)$, we have, by the central limit theorem, the limiting distribution,
    \begin{align*}
        \sqrt{n} \left( x - \frac{1}{p} 1_p \right) \rightsquigarrow \mathcal{N}_p \left( 0 , \frac{1}{p} \left(I_p - \frac{1}{p} 1_p 1_p^{\mathrm{T}} \right) \right).
    \end{align*}
    The function, say $g$, that maps $x$ to the quantity inside the brackets on the left-hand side of \eqref{eq:s2_s3_connection} has the following gradient and Hessian, for $i \neq j$,
    \begin{align*}
        \nabla_i g(x) &= 2 (p + 1) x_i - 3 p(p - 1) x_i^2 + 4 p(p - 2) s_2 x_i \\
        \nabla^2_{ii} g(x) &= 2 (p + 1) - 6 p(p - 1) x_i + 8 p(p - 2) x_i^2 + 4 p(p - 2) s_2 \\
        \nabla^2_{ij} g(x) &= 8 p(p - 2) x_i x_j.
    \end{align*}
    Thus, plugging in $x = (1/p) 1$, the gradient takes the value,
    \begin{align*}
        3 \left( 1 - \frac{1}{p} \right) 1_p,
    \end{align*}
    whereas the Hessian equals,
    \begin{align*}
        8 \left( 1 - \frac{2}{p} \right) 1_p 1_p^{\mathrm{T}}.
    \end{align*}
    Next we observe that the vector $1_p$ belongs to the null space of the matrix $I_p - (1/p) 1_p 1_p^{\mathrm{T}}$, to which the limiting covariance matrix of the vector $x$ is proportional. Consequently, both the first and second-order Delta method applied to the map $g$ produce a degenerate limiting distribution for $g(x)$, finishing the proof.
\end{proof}

\bibliographystyle{apalike}
\bibliography{paper-ref}

\end{document}